\documentclass[bjps,a4paper,11pt,oneside,onecolumn,english,reqno]{imsart}
\usepackage[left=2.5cm,right=2.5cm, top=3cm, bottom=3cm]{geometry}
\usepackage{longtable,array,amsmath,amsfonts,amsthm,amssymb,mathrsfs,color}
\usepackage{amsthm}
\usepackage{enumerate}
\RequirePackage{amsthm,amsmath}
\usepackage{amsmath}   % many want amsmath extensions
\usepackage{amsfonts,amssymb,color}

\newcommand\del[1]{}

\usepackage{url}      % for URLs

\usepackage{babel}
\usepackage{latexsym}
\newcommand\cadlag{c{\`a}dl{\`a}g\,\,}
\newcommand\ind{\rm ind}
\newtheorem{notation}{Notation}[section]

\newtheorem{thm}{Theorem}[section]
\newtheorem{rem}{Remark}[section]
\newtheorem{prop}{Proposition}[section]
\newtheorem{cor}{Corollary}[section]
\newtheorem{ex}{Example}[section]
\newtheorem{defn}{Definition}[section]

\newcommand{\eps}{\varepsilon}

\newcommand\sou[1]{}

\numberwithin{equation}{section}

\newcommand{\bcase}{\begin{cases}}
\newcommand{\ecase}{\end{cases}}
\newcommand{\pmat}{\begin{pmatrix}}
\newcommand{\epmat}{\end{pmatrix}}

\newcommand{\sgn}{\mbox{sgn}}

\newcommand{\levy}{L\'evy }

\newcommand{\barray}{\begin{array}{rcl}}
\newcommand{\earray}{\end{array}}

\newcommand{\lqq}{\lefteqn}

\newcommand{\la} {{\langle}}
\newcommand{\ra} {{\rangle}}

\newcommand{\CE} {{\mathcal{E}}}
\newcommand{\OPER} {{\Theta}}
\newcommand{\CI} {{\mathcal{I}}}
\newcommand{\CBB} {{\mathcal{B}}}

\newcommand{\CL} {{\mathcal{L}}}
\newcommand{\CX} {{\mathcal{X}}}

\newcommand{\CSS} {{\mathcal{S}}}

\newcommand{\lk}{\left}
\newcommand{\rk}{\right}

\newcommand{\ep} {\varepsilon }

\newcommand{\be} {\begin{enumerate} }
\newcommand{\ee} {\end{enumerate} }

\newcommand{\vr} {{ \varrho }}

\newcommand{\CF}{{ \mathcal{ F } }}
\newcommand{\CG}{{ \mathcal{ G } }}
\newcommand{\CA}{{ \mathcal{ A } }}

\newcommand{\CM}{{\mathbb{M}}}
\newcommand{\RR}{{\mathbb{R}}}

\newcommand{\DD}{\mathbb{D}}
\newcommand{\QQ}{\mathbb{Q}}

\newcommand{\NN}{\mathbb{N}}
\newcommand{\MA}{\mathfrak{A}}

\newcommand{\CY}{\mathcal{Y}}

\newcommand{\PP}{{\mathbb{P}}}

\newcommand{\EE}{ \mathbb{E} }
\newcommand{\TT}{{\rm I \kern -0.2em T}}

\newcommand{\DEQS}{\begin{eqnarray*}}
\newcommand{\EEQS}{\end{eqnarray*}}
\newcommand{\DEQSZ}{\begin{eqnarray}}
\newcommand{\EEQSZ}{\end{eqnarray}}

\begin{document}

\begin{frontmatter}

\title{Nonlinear filtering with correlated L\'evy  noise characterized by copulas}
%\runtitle{A Sample Document}

\runtitle{Nonlinear Filtering with correlated L\'evy noise}

\begin{aug}
% indicate corresponding author with \corref{}
% \author{\fnms{John} \snm{Smith}\thanksref{a}\corref{}\ead[label=e1]{smith@foo.com}\ead[label=e2,url]{www.foo.com}}
% \address[a]{\printead{e1};\printead{e2}}

\author{\fnms{B. P. W.} \snm{Fernando}\ead[label=e1]{bandhisattambige.fernando@unileoben.ac.at} }
\and
\author{\fnms{E.} \snm{Hausenblas}\ead[label=ehier]{erika.hausenblas@unileoben.ac.at}}

\address{Lehrstuhl f\"ur Angewandte Mathematik, \\Montanuniversit\"at
  Leoben,\\
Franz Josef Stra\ss e 18, 8700 Leoben, \\Austria,\\
          \printead{e1} }

\address{Lehrstuhl f\"ur Angewandte Mathematik, \\Montanuniversit\"at
  Leoben,\\
Franz Josef Stra\ss e 18, 8700 Leoben, \\Austria,\\
          \printead{ehier} }

%\thankstext{a}{Authors were supported by the Austrian Science Foundations, Project number P23591.}

\affiliation{Montanuniversitaet Leoben}

 \runauthor{B. P. W. Fernando  and E. Hausenblas}

\end{aug}

\del{  \runauthor{E. Hausenblas  et al.}

  \affiliation{Montanuniversitaet Leoben}

  \address{Franz Josefstra\ss e 18, 8700 Leoben, Austria,\\
          \printead{e1,e2}}

\end{aug}
}
\begin{abstract}

The objective in stochastic  filtering  is to reconstruct
the information about an unobserved (random) process, called the
signal process, given the current available observations of a
certain noisy transformation of that process.

Usually $X$ and $Y$ are modeled by stochastic differential equations driven by
a Brownian motion or a jump (or L\'evy) process.
 We are interested in the situation where both
the state process $X$ and the observation process $Y$ are
perturbed by coupled \levy processes.
More precisely,
$L=(L_1,L_2)$ is a $2$--dimensional \levy process in which the
structure of dependence is described by a \levy copula.
We derive the associated Zakai equation for the density process and establish sufficient conditions depending on the copula and $L$
for the solvability of the corresponding solution to the Zakai equation.
In particular, we give conditions of existence and uniqueness of the density process, if one is interested to estimate quantities like
$\PP( X(t)>a)$, where $a$ is a threshold.
\end{abstract}



\end{frontmatter}

         % and the secondary ams number(s) in the second bracket
         % e.g. \ams{60E20}{49G03;49F10}

\date{\today}

\section{Introduction}

The objective in stochastic  filtering  is to reconstruct
information about an unobserved (random) process, called the
signal process, given the current available observations of a
certain noisy transformation of that process.
Here, the underlying problem is, that the unobserved problem may
be corrupted by noise, and in addition, the observations made are
usually again corrupted by some noise or
random errors. The main objective of stochastic filtering is to estimate an
evolving dynamical system usually called signal.
That is, to extract the most precise information
about the underlying system and to
filter out the ``noise'' in the observations.
These kind of problem appears in physics, engineering, and
finance among others.

%\medskip

This measurement noise is modeled very often by a
stochastic process of Gaussian or Poisson type.
In particular, the signal and the
observation process can be modeled  either by a discontinuous or
continuous random process.
When both the signal $X$ and the
observation $Y$ have discontinuous paths, one can distinguish three main frameworks. The first one is the
case in which $Y$ is driven by a counting process or a marked point process.
We can refer to \cite{Bremaud,Ceci-2006,Frey+Runggaldier-2001,Klieman et al-1990,Shiryaev+Lipster-1989}, and
 \cite{Shiryaev+Lipster-2001-II} among others for the results and advances made in
this situation. The second framework is the case in which $Y$ is driven by a
 mixed type process, that it, $Y$ can be viewed as a sum of marked point
process and a diffusion process. This case is the subject of
recent papers   \cite{Ceci+Colaneri-2012-ADV,Frey+Runggaldier-2010,Frey+Schimdt-2012,Frey et al-2011}.
Finally, one can model  the signal $X$ and the observation $Y$  by a jump-diffusion processes, which is done e.g.\ in
 \cite{Ceci+Colaneri-2012-ADV}.
 In that work, they also allow processes $X$ and $Y$ to be
correlated and have common jump times.

 In the present paper we consider the filtering problem similar to the model in \cite{Ceci+Colaneri-2012-ADV} but address the
 difficult situation where the signal and observation process are driven by two \levy processes which are
 correlated. To be more precise, in our model the state $X$
and the observable $Y$ solve a stochastic equation driven by
general \levy processes. The Brownian part in $X$ may be degenerate. In addition both
processes are corrupted by a pair of two purely discontinuous \levy processes,
where the dependence structure is given by a \levy copula. Here
$X$ is corrupted by the first  process and the observation process is corrupted by the
second process.
By using the change of measure method we
derive the associated Zakai equation. Using copula, we were able to calculate the diffusion
coefficient in front of the random driving process in the Zakai
equation explicitly. We treat the case of finite and infinite
\levy measure separately  in Theorem \ref{copula_finite} and
 Theorem \ref{copula_inf}. As mentioned in the abstract, we were mainly interested in the case where one would like to estimate entities like
 $\PP\lk( X(t)>a\mid Y(s),\, 0\le s\le t\rk)=\EE [1_{(a,\infty)}(X(t))\mid Y(s),\, 0\le s\le t]$, $a\in\RR$.
 Here the main difficulty is that the function $\RR\ni x\mapsto 1_{(a,\infty)}(x)$ is not twice differentiable and one has
to use the smoothing property of the infinitesimal generater of
the driving \levy process of $X$ (see \cite{hsym}). Because of this, we also use the change of measure transformation and consider the Zakai Equation.
In this paper, we were able to specify in Theorem \ref{measurevalued}
the exact conditions under which
the density process exists and is uniquely defined. In addition, we investigated the regularity of the process.

The organization of the paper is as follows. In Section 2, we
introduce the problem and derive the Zakai Equation for finite and
infinite \levy measures. In Section 3 we consider the case where
one is interested to estimate an entity like $\PP\lk( X(t)>a\rk)$,
$a\in\RR$. Here, the main result is Theorem \ref{measurevalued}.
Corollary \ref{spec_ex} is an example which illustrates the
applicability of Theorem \ref{measurevalued}. In the appendix we
summarize results that are necessary for the proofs of our main
results. In particular, in \ref{aA} we introduce the
Zakai equation as an evolution equation taking values in Sobolev
spaces. In \ref{aB} we introduce \levy copulas and give
known results necessarily for the proofs of our main results.

%%%%%%%%%%%%%%%%%%%%%%%%%%%%%%%%%%%%%%%%%%%%%%%%%%%%%%%%%%%%%%%%%%%
%%%%%%%%%%%%%%%%%%%%%%%%%%%%%%%%%%%%%%%%%%%%%%%%%%%%%%%%%%%%%%%%%%%%

\begin{notation}
We denote by $\RR_+$ the positive real half line, i.e.\ $\RR_+=(0,\infty)$, and by
$\RR^0_+$ the positive real have line including zero,  i.e.\ $\RR^0 _+=[0,\infty)$.
For  a measurable space $(E,\CE)$ we denote by $B_b(E)$ the Banach space  of all bounded, real--valued, $\CE$--measurable functions
equipped with the supremum norm.  For  a metric space $(E,\CE)$ we denote by $C_b(E)$ the Banach space  of all bounded, real--valued and continuous  functions
equipped with the supremum norm. Let us denote by $\CSS$ the Schwartz space of all  rapidly decreasing functions and $\CSS^\prime$ its dual.
For $s\in\RR$ and $p\ge 1$ we denote by $H_p^s(\RR^d )$ the Bessel Potential Spaces (or Sobolev spaces of fractional order), i.e.\
$$
H^s_p(\RR^d):=\{ f\in \CSS^\prime: |f|_{H^s _p}:= |\CF^{-1}(1+|\xi|^2 )^\frac s2 \CF f|_{L^p}<\infty\}.
$$
Here, $\CF$ denotes the Fourier transform given by
 $$
\mathcal{F} f(\xi) =  \hat f(\xi)=(2\pi) ^{-d} \, \int_{\RR^d} e^ {i\xi ^Tx} f(x)\, dx, \quad f\in L^ 2 (\RR^d).
$$
The space $C^{(n)}_b(\RR)=\{f:\RR\rightarrow\RR:f\;\text{is}\;n\;\text{times}\;\text{continuously}\;\text{differentiable}\;\text{and}\;\text{bounded}\}.$
\end{notation}

\section{Problem setting and the Zakai equation}

As mentioned in the introduction, we consider the filtering problem with \levy noise.
In particular, the state and observation processes are both perturbed by a
\levy noise. Since in practice the noises in the state process and the observation process
are usually depending on each other, so we allow our model to have certain dependence structure.

In the case of Gaussian variables the dependence structure is described via a correlation matrix.
However for the non-Gaussian random variables, the use of correlation coefficients is often misleading. Hence, we must choose
the right tool to describe the dependence structure for non-Gaussian noise.
%The use of copula is one the most popular and practical approach to describe dependence non-Gaussian
%random variables.
Here, copulas are nowadays widely used in finance to express
dependence of non-Gaussian
random variables. In Apendix \ref{levy_copula} we give a short summary on
copula and some facts that we need for the proof of our
main results. For a more detailed  introduction, we refer to the books Cherubini et al.\ \cite{cher}, Nelsen \cite{nelsen},  Malvergne and Sornette \cite{mal}.

Let $(X_1,X_2,\ldots,X_n)$ be a random vector with marginal distribution functions $F_i$, i.e. $F_i(x)=\PP\lk( X_i\le x\rk)$.
By assuming $F_1, \ldots, F_n$ are continuous, one can show that up to a transformation the random vector
$$
( F_1(X_1), F_2(X_2),\ldots, F_n(X_n))
$$
has uniformly  distributed  margins. The cumulative distribution function $
(U_1, U_2,\ldots, U_n)
$ associated to $(X_1,\ldots, X_n)$ is defined by \del{
$$
C:[0,1]^ n \to \RR^ +_0
$$
is defined as  of
$$
(U_1, U_2,\ldots, U_n)
$$
with} $U_i=F_i(X_i)$, $i=1,\ldots, n$.
For any random vector $(X_1,X_2,\ldots,X_n)$  with distribution $F:\RR ^n\to[0,1]$ and continuous marginal distribution functions $F_i$, the function
$$
C:[0,1]^ n \to [0,1]
$$
such that $$C(u_1,\ldots, u_n)= F(F_1 ^{-1}(u_1),\ldots , F_n ^{-1}(u_n)),\quad u_i\in[0,1],\, i=1,\ldots ,n,
$$
is called the copula. The existence of a copula $C$ associated to given
marginal distribution $F_1, \ldots, F_n$ is ensured by following
theorem.
\begin{thm}\label{sklar_theo}[Sklar's  Theorem] Given an  $n$-dimensional distribution
function $F$ with continuous (cumulative) marginal distributions $F_1,\ldots , F_n$,
there exists a unique $n$-copula $C : [0, 1]^n\rightarrow [0, 1]$ such that
$$
F(x_1, \ldots, x_n) = C(F_1(x_1), . . . , F_n(x_n)) ,\quad \forall (x_1,\ldots,x_n)\in\RR^n .
$$
\end{thm}
\del{ I do not understand!!!
In addition one can show that
a copula defined as above satisfies certain properties {WHAT KIND OF PROPERTIES}.
Conversely, if a function satisfy these properties, then it describes
the dependence structure of certain random variables.}
There are several different types of copulas. The ones very
frequently seen in the literature are the independent copula
defined by
$$ C(u_1, u_2, \ldots , u_n) = u_1 u_ 2 \cdots  u_n,
$$
the Clayton copula defined for $\theta \in  [-1,\infty)\setminus\{0\}$  by
$$ C(u_1, u_2, \ldots , u_n) = \max\lk([u_1^{-\theta}+ u_ 2 ^{-\theta} +\cdots+u_n^{-\theta} -(n-1)]^{-\frac 1 \theta},0\rk),\,
$$
and the Gumpel copula defined for $\theta \in  [1,\infty)$  by
\DEQS \lqq{ C(u_1, u_2, \ldots , u_n) }&&\\&=&  \exp\lk(-\lk[
(-\ln u_1)^{\theta}+ (-\ln u_ 2) ^{\theta} +\cdots +(-\ln
u_n)^{\theta} \rk]^{\frac 1 \theta}\rk). \EEQS

In a similar way we can define the L\'evy copulas which is a
general concept to capture jump dependence in multivariate L\'evy
processes. The L\'evy copula  is described in terms of the \levy
measure. For more detailed introduction to \levy copula, we refer
to the works of Cont and Tankov \cite{tankov,tankov1} and Tankov
and Kallson \cite{tankov2}. In addition we summarize some basic facts in appendix \ref{levy_copula}.
Since the L\'evy measure is usually $\sigma$--finite, the
definition of a copula has to be extended to a function acting on
$[-\infty,\infty]$.

For this purpose, let $\nu$ be a \levy measure on $\RR ^n$ with marginal intensities
$\nu_1,\nu_2,\ldots,\nu_n$.
Let $\CI:\RR\setminus\{0\}\to\CBB (\RR)$ be given by
\DEQS
\mathcal{I} (x) =\bcase (x,\infty)\,& x>0,
\\
(-\infty,x), & x<0.\ecase
\EEQS
Let $U_i$ be the tail integral defined by
\DEQSZ\label{tail1}
U_i(z)=\bcase \sgn(z)\nu_i( \CI(z)), & \mbox{for} \, z\in \RR\setminus\{0\}\\
0 & \mbox{for} \, z=\infty \mbox{ or } z=-\infty\\
\infty & \mbox{for} \, z=0,\quad i=1,2,\ldots,n \\
\ecase
\EEQSZ
and
\DEQSZ\label{def_tail_n}
U(z_1,z_2,\ldots ,z_n)=\bcase \lk( \prod_{i=1}^n \sgn(z_i)\rk)\,  \nu\lk( \prod_{i=1}^n \CI(z_i)\rk), &
\\
\quad  \mbox{for} \, z_1,z_2,\ldots ,z_n\in \RR\setminus\{0\}
\\0, \quad \mbox{for} \, |z_i|= \infty,i=1,\ldots, n\\
\nu(\RR^n), \quad  \mbox{for} \, z_i= 0,\,  i=1,\ldots,n. \ecase
\EEQSZ Now, for an $n$--dimensional L\'evy process $L$, one can
associate a \levy copula $ H:[-\infty,\infty] ^n
\to[-\infty,\infty] $ as
$$
U(z_1,\ldots, z_n) = H(U_1(z_1),\ldots, U_n(z_n)),\quad z_1,\ldots,z_n\in\RR.
$$
In fact, thanks again to Sklar--type Theorem (see \cite[Theorem
3.6]{tankov2}) for each $n$--dimensional
\levy process with intensity $\nu$ and  marginal intensities % $\nu_i(z_i,\infty)=\nu\lk( (-\infty,\infty)\times \cdots \times (z_i,\infty)\times \cdots  \times (-\infty,\infty)\rk)$, $i=1,\ldots,n$, $z_i>0$ and
$\nu_i$, $i=1,\ldots,n$,  %(z_i,\infty)=\nu\lk( (-\infty,\infty)\times \cdots \times (-\infty,z_i)\times \cdots  \times (-\infty,\infty)\rk)$, $z_i<0$,
there exists  a \levy copula $H$ such that
\DEQSZ\label{eqn1}
U(z_1,\ldots, z_n) = H(U_1(z_1), \ldots ,U_n(z_n)),\quad z_1,\ldots, z_n\in \RR. %\ge 0.
\EEQSZ
%In \ref{levy_copula}, we recall the basic facts about \levy copulas needed for ourpurpose.
%For more details about \levy copulas we refer the interested reader to the book \cite{tankov} and the articles \cite{tankov1,tankov2}.

\medskip

Now, let us proceed with the setting of our main problem. Let $H$
be a \levy copula and $L=\{ L(t)=(L_1(t),L_2(t))\in\RR^2 \, :t\ge
0\}$ be a two dimensional
pure jump  L\'evy process with its marginal intensities $\nu_1$ and $\nu_2$. %, where $\nu_1$ and $\nu_2$ are supported by $\RR_+^0$.
Let $L_0$ be a compensated pure jump  \levy process  and $W_2=\{W_2(t):t\ge 0\}$ be a Brownian motion. We assume that
all these objects are defined on a probability space $\MA=(\Omega,\CF,(\CF_t)_{t\ge 0},\PP)$.
We also assume that $L$,
 $L_0$ and $W_2$ are mutually  independent.

\medskip

 Let the signal process  $X$ be the solution of the following SDE with random initial data $X_0$:
 \DEQSZ\label{eqn-x}
\lk\{ \barray  dX(t) &=& b(X(t))\, dt + dL_0(t)+ d L_1(t),\quad t>0,
\\ X(0)&=& X_0.\earray\rk.
\EEQSZ
Here $b:\RR\to\RR$ is a Lipschitz continuous function. Also we
suppose that the observable process $Y$ solves the following SDE with random initial data $Y_0$.
 \DEQSZ\label{eqn-y}
\lk\{ \barray  dY(t) &=& %\int_0^ t
g(X(t) )\, dt + dL_2(t)+ dW_2(t) ,\quad t>0,\\ Y(0)&=&
Y_0.\earray\rk.
 \EEQSZ
where $g:\RR\to\RR$ is a twice differentiable mapping. Let $\{
\CX_t: t\ge 0\}$ and $\{ \CY_t: t\ge 0\}$ be\del{the usual
augmentation of} the filtration \del{of $X$ and $Y$, respectively,
} defined by $\CX_t=\sigma(\{ X(s), s\le t\})$ and
$\CY_t=\sigma(\{ Y(s), s\le t\})$, respectively. In addition, let
$\CX=( \cup_{t\ge 0} \CX_t)$ and $\CY=( \cup_{t\ge 0}
\CY_t)$.

The filtering problem consists of determining at a fixed time $t>0$
the conditional distribution $\pi_t$ of the signal $X$ given the information accumulated
from observing $Y$ in the time interval $[0,t]$; that is, for $f\in C ^{(2)}_b(\RR)$, we are aiming to compute the Bayes estimator %for a given functions $f\in C ^{(2)}_b(\RR)$
$$
\pi_t( f) =
\EE \lk[ f(X(t))\mid \CY_t\rk], \quad t\ge 0.
$$
%%p.48 crisan
\del{The Bayes estimator of the density process is then given by
$$
\pi_t f = {\EE \lk[ f(X(t))\mid \CY_t\rk]\over \EE \lk[ f(X(t))\mid \CY_t\rk]}, \quad t\ge 0.
$$
}

In order to study about the normalized conditional density $\pi=\{ \pi_t:t\ge 0\}$,
one can mainly use two different methods. The first one is probability measure transformation and obtain Zakai equation which solves the un-normalized conditional density associated with normalized density $\pi$. Then discuss about $\pi$ using Kallianpur-Striebel formula (see \cite[Proposition 3.16]{BaC09}). The second method is called innovation approach which directly gives Fujisaki-Kallianpur-Kunita equation (called "FKK equation").  Normalized density $\pi$ is the solution of FKK equation. In this paper we use the former method.

In the first step we apply the Girsanov's Theorem to get a new
measure $\QQ$ which is chosen in such a way that $Y$ is a
\levy process over the probability space
$(\Omega,\CY,(\CY_t)_{t\ge 0},\QQ)$.
% under which $Y$ is a independent \levy process{independent to which process}.
For this purpose let $Z=\{Z(t):t\ge 0\}$ be given by
\DEQSZ\label{zsolves} Z(t)
&:=& \exp\lk( - \int_0^t g(X(s))\, dW_2(s)-\frac 12 \int_0^t g^2
(X(s))\, ds\rk), \quad t\ge 0.
\EEQSZ \sou{crisan, p. 52 (3.18)}
Note, that $Z$ solves
\DEQS \lk\{\barray dZ(t) &=& Z({t^-})
g(X({t^-}))\, dW_2(t),\\ Z(0)&=&1. \earray\rk. \EEQS over
$(\Omega,\CF,(\CF_t)_{t\ge 0},\PP)$. Let $\QQ$ be a new probability measure given by
\DEQSZ\label{defineQ} { d \QQ\over d\PP}\Big|_{\CF_t}&=&
Z(t),\quad t\ge 0. \EEQSZ
As in the Brownian case, one can show the following proposition.
\begin{prop}\label{above}
If
\DEQS
\EE \lk[ \int_0^t \| g(X(s))\| ^2 \, ds \rk] <\infty,\quad
\EE \lk[ \int_0^t Z(s)\, \| g(X(s))\|  \, ds \rk] <\infty,\quad
t\ge 0,
\EEQS
then under $\QQ$ the observation process $Y$ is a L\'evy process.
In particular, the $\sigma$-field $\CY ^+_t=\sigma( Y(r)-Y(s), t\le s\le r)$ is independent to $\CY_t$.
\end{prop}

\begin{proof}
Let $\QQ$ be defined as in equation \eqref{defineQ}. Firstly, note that by the It\^o-L\'evy
decomposition the continuous and discontinuous parts of $Y$ are independent. In addition, under the new probability measure $\QQ$, the continuous part of $Y$ is a Brownian motion. We can also see that the pure jump process is not affected by the change of measure.
\end{proof}
{Setting $V(t)=Z(t)^{-1}$, we obtain as
in \cite[Eq. (3.30) page 56]{BaC09} that \DEQS { d \PP\over
d\QQ}\Big|_{\CF_t}&=& V(t),\quad t\ge 0. \EEQS
\begin{rem}
The process $V=\{V(t):t\ge 0\}$ defined by $V(t)=Z(t)^{-1}$ solves
on $(\Omega,\CF,\PP)$ the equation
\DEQSZ\label{eqn-vnoep}
\lk\{\barray
dV(t) &{=}& V({t}) g(X({t}))\,\lk[  dW_2(t) +g(X(t))\, dt \rk]
\\ & =& V({t}) g(X({t}))\, dY^ c(t),\\ V(0)&=&1.
\earray\rk.
\EEQSZ
($Y^c$ denotes the continuous part of $Y$, i.e.\ the part of $Y$ without jumps).
%on $(\Omega,\CF,\PP)$.
Since the process $ W_2(t) +\int_0^t g(X(s))\, ds$ becomes a Brownian motion over $(\Omega,\CF,(\CF_t)_{t\ge 0},\QQ)$,
$V$ is a $(\Omega,\CF,(\CF_t)_{t\ge 0},\QQ)$--martingale.
\end{rem}}
The following result is an immediate consequence of Proposition \ref{above}. We also refer to \cite[Proposition 3.15, page 56]{BaC09}.
\begin{cor}\label{cor23}
If $U$ is $\CF_{t^-}$--measurable,  %(or $\CF_{t}$--measurable),
then the law of the two random variables
$
\EE^ \QQ \lk[ U \mid \CY \rk]$ and $\EE^ \QQ \lk[ U \mid \CY_{t^-}\rk]
$ are the same over $(\Omega,\CF,\QQ)$.
%(or $ \EE^ \QQ \lk[ U \mid \CY \rk]$ and $\EE^ \QQ \lk[ U \mid \CY_{t}\rk] $ are the same over  $(\Omega,\CF,\QQ)$).
In particular, we have $\QQ$-a.s.\ %on $(\Omega,\CF,\QQ)$ we have
$$
\EE^ \QQ \lk[ U \mid \CY \rk] %\stackrel{d}
{=} \EE^ \QQ \lk[ U \mid \CY_{t^-}\rk].
$$
\end{cor}
\begin{rem}\label{rem2_2}
Similarly it can be shown that
if $U$ is $\CF_{t}$--measurable,
then $\QQ$-a.s.\ %on $(\Omega,\CF,\QQ)$ we have
$$
\EE^ \QQ \lk[ U \mid \CY \rk] %\stackrel{d}
{=} \EE^ \QQ \lk[ U \mid \CY_{t}\rk].
$$
\end{rem}
\begin{proof}
Since $Y$ is a \levy process over $(\Omega,\CF,(\CF_t)_{t\ge 0},\QQ)$, its
increments are independent. Hence, for all $t>0$, the $\sigma$--algebra
$\CY_{t^-}^ +$ generated by $Y(s)-Y({t^-})$, $s > t$ is
independent to  $\CY_{{t^-}}$ under the measure $\QQ$. From
\cite[Proposition 6.6, page\ 110]{kallenberg} the assertion follows.
\end{proof}
Fix $t\ge 0$. Let $\pi_t$ be the conditional distribution of $X(t)$ at time $t\ge 0$. The Kallianpur-Striebel formula  gives
for $t\ge 0$ (see \cite[Proposition 3.16]{BaC09})
$$
\pi_t(f)=\EE \lk[ f(X(t))\mid \CY_t\rk] = \int_\RR \pi_t (x)\, f(x)\, dx={\EE  ^\QQ\lk[ f(X(t)) V(t)\mid \CY_t\rk]\over \EE ^\QQ\lk[V(t)\mid \CY_t\rk]
}.
$$
Now, we introduce the density process of the un-normalized
conditional distribution $\rho=\{\rho_t:t\ge 0\}$ which is the measure valued
process defined by
$$
\rho_t(f)=\la \rho_t,f\ra=\EE ^ \QQ  \lk[ V(t)\, f(X(t))\mid\CY_t\rk]  = \int_\RR \rho_t (x)\, f(x)\, dx,\quad t> 0, \quad \rho_0=\pi_0.
$$
We will see later on, that the process $\rho=\{\rho_t:t\ge 0\}$ is very useful to calculate  $\pi=\{\pi_t:t\ge 0\}$.

By Corollary \ref{cor23}, we have %over  $(\Omega,\CF,\QQ)$
$$
\EE^ \QQ[f(X(t))V(t)\mid\CY] %\stackrel{d}
{=}  \EE^ \QQ[f(X(t))V(t)\mid\CY_t] =\la \rho_t,f\ra,\quad t\ge 0,\quad\QQ-a.s..
$$
We also introduce the process $\xi=\{\xi(t):t\ge 0\}$ defined by
\DEQSZ\label{xiintro}
\xi(t)=\EE^\QQ\lk[ V(t)\mid \CY_t\rk],\quad t\ge 0.
\EEQSZ
Since $V$ is a $\CF_t$--martingale over $(\Omega,\CF,\QQ)$ and $\CY_t\subset\CF_t
$, it follows that for $0\le s<t$
$$
\EE^ \QQ [\xi(t)\mid \CY_s]  =\EE^ \QQ\lk[ \EE^ \QQ [ V(t)\mid \CF_s]\mid \CY_s\rk] = \EE^ \QQ[ V(s)\mid \CY_s]=\xi(s).
$$
Moreover,
$$
\xi(t) \pi_t(f)=\rho_t(f),\quad t\ge 0,
$$ and
$$
 \pi_t(f)=\rho_t(f)\xi^ {-1}(t),\quad t\ge 0.
$$
For these two formulas, we refer to \cite[Definition 3.17 \& Corollary 3.19, pages\ 58-59]{BaC09}.
\del{Observe, $\xi$ solves
\DEQS
\xi(t)= 1+\int_0^t \EE^  \QQ [ V(s^-)g(X(s^-))\mid \CY_{s^-}]\, dY_s^c,\quad t\ge 0.
\EEQS
The inverse $\varsigma=\{ \varsigma_t:t\ge 0\}$ is given by}

In the next theorem, we will derive the Zakai equation which is solved by the un-normalized density
process $\rho=\{ \rho_t:t\ge 0\}$. To do that, we need to
introduce some additional notations. A \levy process $L$ is characterized by
 its
characteristic function. In particular, there exists a function $\psi:\RR \mapsto \mathbb{C}$ such that
$$
\ln(\EE e^{i\xi L(t)}) = t\psi(\xi),\quad \xi\in\RR.
$$
The infinitesimal generator of the Markovian semigroup of $L$ is the so called pseudo--differential operator
given by %let us define the to the function %symbol %\footnote{For a short introduction for symbol, see \cite{all}, where we also have proven same estimates we needed for this article.}
%$\psi$ associated pseudo differential operator   $A_0$  by
%
\DEQSZ\label{pseudo_def} A_0\, f := - \int_\RR e^{i \xi x }
\psi(\xi) \CF f (\xi)\, d\xi, \quad f\in C^{(2)}_b(\RR). \EEQSZ
Here $\CF f$ denotes the Fourier transform of the function $f$. The
function $\psi$ is called the \levy symbol of the \levy process
$X$, for more details on $A_0$ and its properties we refer to \cite{hsym}.
The following theorem associates with the case where the \levy measure of the two dimensional \levy process $L$ is finite.
%{assumption on $h$ and $\sigma$}
\begin{thm}\label{copula_finite}
Let $L_0$ be a \levy process with \levy symbol $\psi$ and $A_0$ be the infinitesimal generator of $L_0$.
Let $\nu_1$ and $\nu_2$ be two finite \levy measures defined on the positive half real line, i.e.\ on $\RR_+$.
Let $H$ be a twice differentiable copula. Let us denote  the conditional \levy measure of jumps of $L_1$ given the jumps of $L_2$  by
\DEQS
\nu_{1,{z_2}}(U) &=& \int_U h(z_1,z_2)\, \nu_1(dz_1),
\quad U\in\CBB (\RR_+),
\EEQS
where
$$
h(z_1,z_2):=
{ {\partial ^ 2 \over \partial u_1\partial u_2} H(u_1,u_2) \Big|_{u_1=U_1 (z_1)\atop u_2=U_2 (z_2)}
 },
$$
and $U_1$, $U_2$ are the tail integrals of $\nu_1$ and $\nu_2$, respectively.
Let $g:\RR\to\RR$ and $\sigma:\RR\to\RR$ be Lipschitz continuous mappings.
%For a given function $f\in C ^{(2)}_b(\RR)$ t
Then the un-normalized conditional density estimator $\rho=\{ \rho_t:t\ge 0\}$ is a solution to the following equation
\DEQSZ\label{variational}
 \la \rho_t,f\ra  &=& \la \rho_0,f\ra +  \int_0^ t  \la
\rho_{{s^-}}\, ,f\cdot g \ra \, dY_s^ c
\\\nonumber  && {}+ \int_0^t \la \rho_{{s^-}},\CA_0  f\ra\, ds
 +\int_0^t   %\int_\RR
 \la \rho_{{s^-}}, \OPER _{z_2}  f\, \ra\, \eta_2(dz_2,ds) %\nu_{1,\Delta L_2(s)}(dz_1)\,
,\quad \forall f\in C ^{(2)}_b(\RR),
\EEQSZ
where $\eta_2$ denotes the Poisson random measure associated to $L_2$ with intensity $\nu_2$,
%$\Delta L_2(t)$ denotes the jumps of $L_2$, i.e.\ $\Delta L_2(t)=\lim_{s\uparrow t} L_2(t)-L_2(s)$, %for $z_2\in\RR$
the operators  $\OPER _{z}$ and $\CA_0$ are defined by
$$\OPER _{z} f(x)=   \int_{\RR_+}
\lk[f(x+z_1)-f(x) \rk] \nu_{1,z}(dz_1),\quad z\in\RR_+, x\in
\mathbb{R}, f\in C ^{(2)}_b(\RR),
$$
and
$$
\CA_0 f (x)=  b(x)f'(x) + A_0 f(x), \quad x\in \mathbb{R}, f\in C
^{(2)}_b(\RR),
$$
where the operator $A_0$ is the infinitesimal generator of the Markovian semigroup of $L_0$ which is a pseudo--differential operator and defined through $\eqref{pseudo_def}$.
%with
%$$
%A_0\, f = - \int_\RR e^{i \xi x } \psi(\xi) \hat f (\xi)\, d\xi, \quad f\in C^{(2)}_b(\RR).
%A_0 f = \int_\RR \lk( f(x+y)-f(x)-f'(x)y\rk) \nu(dy), \quad f\in C ^{(2)}_b(\RR).
%$$
%%
\end{thm}
\del{
\begin{rem}
The operator $A_0$ is the infinitesimal generator of the Markovian semigroup of the \levy process, see appendix \ref{symbol} for a more detailed description.
\end{rem}
}
\begin{rem}
Since $\nu_1$ and $\nu_2$ are finite \levy measures, the operator
$\OPER _z:H_2^s(\RR)\to H_2^s(\RR)$ is bounded for all $z\in\RR$
and $s\in\RR$. This can be seen by analyzing the symbol $\phi_z$
associated to $\OPER_z$ defined as
$$
 \phi_{z}(\xi)=  \int_{\RR_+} \lk( e^ {iz_1\xi} -1\rk)\, h(z_1,z)\,\nu_1(dz_1).
 $$
In fact, calculating the modulus of the symbol $\phi_z$
$$\lvert \phi_{z}(\xi)\rvert := \lk|\int_{\RR_+} \lk( e^ {iz_1\xi} -1\rk)\, h(z_1,z)\,\nu_1(dz_1)\rk| \le 2\, \int_{\RR_+} \, |h(z_1,z)|\,\nu_1(dz_1)<\infty,
$$
we see that $|\phi_{z}(\xi)|\le C$ for all $\xi\in\RR$. Therefore, $ \Phi_{z}:L^2(\RR)\to L^2 (\RR)$ defined by
$$
\lk( \Phi_z u \rk)(\xi):=\phi_z(\xi)\, u(\xi),\quad \xi\in \RR, \,\, u\in L^2 (\RR),
$$
is a bounded operator.
Using the
spectral Theorem (see e.g.\ \cite[Theorem 4.9, p.\ 30]{engel}) one
sees, that $\Phi_{z}$ acting on $L^2(\RR)$ as a multiplication operator corresponds via the
Fourier transform to $\OPER _z$ acting on $L^2(\RR)$.
Next, the operator $\CF^ {-1} (1+|\xi|^ 2 )^ \frac s2\CF$ is an isometry
from $H^ s_2(\RR)$ to $L^ 2(\RR)$.
Hence,  $\Phi_{z}$ is also bounded on
$H_2^s(\RR)$.
This
implies that $\OPER_z:H_2^s(\RR)\to H_2^s(\RR)$ is bounded for all
$z\in\RR$ and $s\in\RR$.
%$ corresponds in the Fourier space to a bounded multiplication operator which is finite.
\end{rem}
\begin{proof}
Let  $\lambda_1=\nu_1(\RR_+)$ and $\lambda_2=\nu_2(\RR_+)$.
Next, let us denote the number of jumps of $L_2 $ in the time interval $[0,t]$ by $N(t)$, the jumps themselves  by
$\{ Y_{2,i}:i=1,\ldots, N(t)\}$ and the jump times by $\{ s_i: i=1,\ldots,{N(t)}\}$.
Then, given the jumps of $L_2$ in the time interval $[0,t]$, $L_1(t)$ can be represented by
$$
L_1(t) = \sum_{i=1}^ {N(t)} Y_{Y_{2,i}}^ {1,i},\quad t\ge 0,
$$
where for $z\in\RR\setminus \{0\}$ the random variable  $Y^ {1}_z$ is distributed as $\nu_{1,z}/\lambda_{1,z}$, $\lambda_{1,z}=\nu_{1,z}(\RR^+)$. {More rigorously, conditioned on the jumps of $L_2(t)$, $L_1(t)$ can be viewed as a compound Poisson process having same jump times of $L_2(t)$ and the size of each jump $Y^ {1,i}$ of $L_1(t)$ depends on the size of the jump $Y_{2,i}$ at time $s_i$.

By conditioning the process $L_1$ given $L_2$},  we can write
\DEQSZ\label{jumps_re}
\nonumber %\hspace{2cm}
&&f(X (t)) = f(X _0) + \int_0 ^t\lk( \CA_0\,f\rk) (X (s))\, ds + M(t)
\\ && {} +  \sum_{1\le i\le N (t)} f( X ({s_i ^-}) +  Y_{Y_{2,i}}^ {1,i}) -f( X ({s_i ^-}) )
\nonumber\\&=& f(X _0) + \int_0 ^t\lk( \CA_0\,f\rk) (X (s))\, ds + M(t)
\nonumber\\ && {} +  \int_0^t \int_{\RR^+_0} \int_{\RR^+_0} \lk[f( X ({s ^-}) +z_1) -f( X ({s^-}) )\rk]\nu_{1,z_2}(dz_1)\, \eta_2(dz_2,ds)
\nonumber\\
&&{}+\sum_{1\le i\le N (t)}f( X ({s_i ^-}) +  Y_{Y_{2,i}}^ {1,i}) -f( X ({s_i ^-}) )\\
&&{}-\int_0^t \int_{\RR^+_0} \int_{\RR^+_0}\lk[f( X ({s ^-}) +z_1) -f( X ({s^-}) )\rk]\nu_{1,z_2}(dz_1)\, \eta_2(dz_2,ds)
\nonumber\\
\nonumber
&=& f(X _0) + \int_0 ^t\lk( \CA_0\,f\rk) (X (s))\, ds + M(t)+ \tilde M(t)
\nonumber\\
\nonumber
&&{}+\int_0^t \int_{\RR^+_0} \int_{\RR^+_0}\lk[f( X ({s ^-}) +z_1) -f( X ({s^-}) )\rk]\nu_{1,z_2}(dz_1)\, \eta_2(dz_2,ds)
,
\EEQSZ
where
$$ M(t)=\int_0^t f'(X(s))\, dL_0(s),\quad t\ge 0,
$$
and
\DEQS
\lqq{ \tilde M(t)=J(t)-R(t)= \sum_{1\le i\le N (t)}  f( X ({s_i ^-}) +  Y_{Y_{2,i}}^ {1,i}) -f( X ({s_i ^-}) ) }
&&\\
&&{}-  \int_0^t \int_{\RR^+_0} \int_{\RR^+_0} \lk[f( X ({s ^-}) +z_1)
-f( X ({s^-}) )\rk]\nu_{1,z_2}(dz_1)\, \eta_2(dz_2,ds) ,\quad t\ge 0.
\EEQS
Since $L_0$ be a compensated pure jump  \levy process, the process  $M=\{M(t): t\geq 0\}$ is a martingale over $(\Omega,\CF,(\CF_t)_{t\ge 0},\QQ)$. First, observe that we can write for a function $\phi$ %represent $L_t^2$ as the following sum
$$
\sum_{i=1}^ {N(t)} \phi(Y_{2,i})= \sum_{i=1}^ {N(t)} \int_{\RR^+} \phi(z_2) \eta_2(dz_2,\{s_i\}).
$$
In addition, we have by the tower property
\DEQSZ\label{newmar}
&&
 \EE\lk[J(t)\,\Big|\, k=N(t), (z_{2,1},\ldots,z_{2,k})=(Y_{2,1},\ldots,Y_{2,k})\rk]
 \\&=&
 \nonumber
 \EE\lk[\sum_{1\le i\le k}  f( X ({s_i ^-}) +  Y_{Y_{2,i}}^ {1,i}) -f( X ({s_i ^-}) )\,\Big|\, k=N(t), (z_{2,1},\ldots,z_{2,k})=(Y_{2,1},\ldots,Y_{2,k})\rk]
\\&=&
\nonumber
\EE\lk[ \sum_{1\le i\le k} \EE\lk[   f( X ({s_i ^-}) +  Y_{Y_{2,i}}^ {1,i}) -f( X ({s_i ^-}) )\,\big|\, Y_{2,i}= z_{2,i} \rk]\,\Big|\, k=N(t)\rk]
\\&=&
\nonumber
 \EE\lk[ \sum_{1\le i\le k}  \int_{\RR^+} \lk\{  f( X ({s_i ^-}) + z_1 ) -f( X ({s_i ^-})) \rk\} \nu_{1,z_{ 2,i}} (dz_1)\,\Big|\,  k=N(t)\rk].
%,\ldots,Y^ {2}_k)=(z_1^2,\ldots,z_k^2)\rk]\rk]
\EEQSZ
Using the representation above, we get
\DEQS
\ldots &=&
\sum_{i=1}^ {N(t)} \int_{\RR^ +}  \int_{\RR^ +} \lk\{  f( X ({s_i ^-}) + z_1 ) -f( X ({s_i ^-})) \rk\} \nu_{1,z_2} (dz_1) \eta_2(dz_2,\{s_i\}).
\EEQS
Replacing the summation by the integral with respect to the time we get
\DEQS
\ldots &=& \int_0^t \int_{\RR^+_0} \int_{\RR^+_0} \lk[f( X ({s ^-}) +z_1)
-f( X ({s^-}) )\rk]\nu_{1,z_2}(dz_1)\, \eta_2(dz_2,ds)=R(t)
.\EEQS
Now we want to show that $\EE^\QQ\lk[\tilde M(t)\mid \CY_t\rk]=0$, $t\ge 0$. Fix $t\ge 0$. Then
\DEQSZ\label{mart}
\EE^\QQ\lk[\tilde M(t)\mid \CY_t\rk]&=&\EE^\QQ\lk[J(t)-R(t)\mid \CY_t\rk]=\EE^\QQ\lk[J(t)\mid \CY_t\rk]-\EE^\QQ\lk[R(t)\mid \CY_t\rk]
\nonumber\\&=&\EE^\QQ\lk[\EE^\QQ\lk[J(t)\mid\CF_{1}\rk]\mid \CY_t\rk]-\EE^\QQ\lk[R(t)\mid \CY_t\rk]
\nonumber\\&=&\EE^\QQ\lk[R(t)\mid \CY_t\rk]-\EE^\QQ\lk[R(t)\mid \CY_t\rk]=0,
 \EEQSZ
where $\CF_{1}=\{k=N(t), (z_{2,1},\ldots,z_{2,k})=(Y_{2,1},\ldots,Y_{2,k}):k\in\mathbb{N}\}\subseteq \CY_t$.
 \del{This result implies that $\tilde{M}=\{ \tilde{M}(t):t\ge 0\}$ is  a martingale on $(\Omega,\CF,\CF_t,\QQ)$. In addition, since $L$ is a \levy process one can show that under $\QQ$, by applying Corollary \ref{cor23} and Remark \ref{rem2_2},
$\EE^\QQ\lk[
M(t)\mid \CY\rk]=0$, $t\ge 0$.
}
%M=\{M(t):t\ge 0\}$ is a $\QQ$--martingale.
Under the new probability measure $\QQ$, the process $V=\{ V (t):t\ge 0\}$ solves the following SDE
$$
dV (t) = V (t)\,g(X(t))\,  dY^ {c} (t),\, \quad t>0, \quad V (0)=1,
$$
where $Y^ {c} $ denotes the continuous part of $Y$ which is a Brownian motion under $\QQ$, adapted to $(\CY_t)_{t\ge 0}$.
Since $V$ is driven by the continuous part of $Y$, and $L_0$ independent from $W_2$, no correlation terms involving the process $V$ appears.
Thus, we get
\DEQS
\lqq{
%\EE ^ \QQ[ f(X^  \ep(t))] =\EE^ \PP\lk[
 f(X (t))\, V (t)=  f(X_0)+ \int_0^ t  V ({{s^-}})\,dM (s)+ \int_0^ t  V ({{s^-}})\,d\tilde{M} (s)
 %\mid \CY(t) \rk]
  }
&&
\\
&+& %\int_0^t f(X^  \ep_s)\, dV^  \ep_s +
 \int_0^t \int_{\RR^+_0}  V ({{s^-}})\, \int_{\RR^+_0} \lk[ f( X ({s^-}) + z_1) -f( X ({s^-}) )\rk] \, \nu_{1, z_2} (dz_1)\, \eta_2(dz_2,ds)
%\underbrace
\\ &+&{}
\int_0^ t V ({{s}})\, g(X ({{s}}))\, f(X ({{s}}))\, dY^ {c}(s)
 + \int_0^ t V ({{s}}) \lk( \CA_0 f\rk) (X ({{s}})) \,
 ds
 .
\EEQS
%
%Here, $\eta_2$ is the to $L_2$ corresponding Poisson random measure.
Taking into account that $M$ is a martingales over  $(\Omega,\CF,(\CY_t)_{t\ge 0},\QQ)$ with \eqref{mart} and taking conditional expectation together with the Fubini Theorem \cite[Theorem 1.1.8]{applebaum} to the entity above, we get
\DEQS
\lqq{
 \EE^ {\QQ} \lk[ f(X (t)) V (t)\mid \CY_t\rk] =  \EE^ {\QQ} \lk[f(X_0)\mid \CY_0\rk]+ \underbrace{\EE^ {\QQ} \lk[ \int_0^ t  V ({{s^-}})\,dM (s)\mid \CY_t\rk]}_{=0} }
&&
\\
& &{}
 + \underbrace{\EE^ {\QQ} \lk[ \int_0^ t  V ({{s^-}})\,d\tilde{M} (s)\mid \CY_t\rk]}_{=0} +\EE^ {\QQ} \lk[ \int_0^t V ({{s}})  \, \lk( \CA_0 f\rk) (X ({{s}}))\, ds\mid \CY_t\rk]
  \\ && + \EE^ {\QQ} \lk[ \int_0^ t V ({{s}})\,  g(X ({{s}}))\,f(X ({{s}}))\, dY^c_s\mid \CY_t\rk]+ \EE^ {\QQ}\lk[  \int_0^t  \int_{\RR^+_0}\rk.
\\ && \lk.V ({{s^-}})\, \int_{\RR^+_0}\lk[ f( X ({s ^-}) + z_1) -f( X ({s ^-} ))\rk]  \, \nu_{1,z_2} (dz_1)\eta_2(dz_2,ds) \mid \CY_t\rk]
\EEQS
 \DEQS
& =&
 f(X_0) +   \int_0^ t\EE^ {\QQ} \lk[ V ({{s}} ) \lk( \CA_0 f\rk) (X ({{s}}))\mid \CY_t\rk]\, ds+\EE^ {\QQ}\lk[ \int_0^t  \int_{\RR^+_0}\rk.
 \\
&& {} \lk. V ({{s^-}})\,\int_{\RR^+_0}
 \lk[ f( X ({s ^-}) + z_1) -f( X ({s ^-} ))\rk]  \, \nu_{1,z_2} (dz_1)\eta_2(dz_2,ds) \mid \CY_t\rk]
 \\
&&{}+ \EE^ {\QQ} \lk[ \int_0^ t V ({{s}})\, g(X ({{s}}))\,f(X ({{s}}))\, dY^c_s\mid \CY_t\rk]
. \EEQS
By imitating the calculations \eqref{newmar} and  \eqref{mart} for $\int_0^ t  V ({{s^-}})\,d\tilde{M} (s)=\int_0^ t  V ({{s^-}})\,dJ (s)-\int_0^ t  V ({{s^-}})\,dR (s)$, we could show that
$$
\mathrm{E}^ {\mathrm{Q}}\big[ \int_0^ t  V ({{s^-}})\,d\tilde{M }(s)| \mathcal{Y}_t\big]=0.
$$
In the next step  we show that
$$\EE^ {\QQ} \lk[ \int_0^ t  V ({{s^-}})\,dM (s)\mid \CY_t\rk]=0.
$$
Since $\int_0^ t  V ({{s^-}})\,dM (s)$ is $\CF_t$-measurable, it follows from  Remark 2.2
\begin{equation}
\begin{split}
\label{zero1}
\EE^ {\QQ}\big[ \int_0^ t  V ({{s^-}})\,dM (s)| \mathcal{Y}_t\big]=\EE^ {\QQ}\big[ \int_0^ t  V ({{s^-}})f'(X(s^-))\, dL_0(s)| \mathcal{Y}_t\big]
\\=\EE^ {\QQ}\big[ \int_0^ t  V ({{s^-}})f'(X(s^-))\, dL_0(s)| \mathcal{Y}\big].
\end{split}
\end{equation}
By following to  \cite[p.\ 60,
the proof of the part (ii) of Lemma 3.21 ]{BaC09} similar arguments we get
\DEQSZ
\label{zero2}
&&\EE^ {\QQ}\big[ \ep_t\EE^ {\QQ}\big[ \int_0^ t  V ({{s^-}})f'(X(s^-))\, dL_0(s)| \mathcal{Y}\big]\big]=\EE^ {\QQ}\big[ \ep_t \int_0^ t  V ({{s^-}})f'(X(s^-))\, dL_0(s)\big]
\nonumber
\\&=&\EE^ {\QQ}\big[ \int_0^ t  V ({{s^-}})f'(X(s^-))\, dL_0(s)\big]
\nonumber
\\&+&\EE^ {\QQ}\langle\int_0^ t i\ep_sr_s\, dY^c(s),\int_0^ .V ({{s^-}})f'(X(s^-))\, dL_0(s)\rangle_t
\nonumber
\\&=&\EE^ {\QQ}\big[ \int_0^ t  V ({{s^-}})f'(X(s^-))\, dL_0(s)\big]
\nonumber
\\&+&\EE^ {\QQ}\int_0^ ti\ep_sr_s\int_0^ .V ({{s^-}})f'(X(s^-))\langle\, dY^c(s),\, dL_0(s)\rangle_t
=0,
\EEQSZ
where $\ep_t=1+\int_0^ ti\ep_sr_s\, dY^c(s)$ is a member of the total set define in \cite[p.\ 355,
(B.19)]{BaC09}. This implies that $\EE^ {\QQ} \lk[ \int_0^ t  V ({{s^-}})\,dM (s)\mid \CY_t\rk]=0$ for any $s\in [0,t]$.
 Since $ X ({{s}})$ and $ V({{s}})$ are  $\CF_{{s}}$--measurable  we have
 \DEQS
 \EE ^ {\QQ} \lk[ V ({{s}}) \,\lk( \CA_0 f\rk) (X ({{s}}))\mid \CY_t\rk]
 &
% \stackrel{d}
{=}&\EE ^{\QQ} \lk[ V ({{s}}) \,\lk( \CA_0 f\rk) (X ({{s}}))\mid \CY_{{s}}\rk].
\EEQS
 Since $Y^ {c}(t)$ is $\CY_t$-measurable and is a $\QQ$--Brownian motion, it follows from \cite[Lemma 1.2]{Bor89} and Corollary \ref{cor23}
 \DEQSZ
\lqq{ \EE^ {\QQ}\lk[ \int_0^ t V ({{s}})g(X ({{s}}))\, f(X ({{s}}))\,dY^c_s\mid \CY_t\rk]}
 \nonumber
\\
&=&\int_0^ t  \EE^ {\QQ} \lk[V({{s}})g(X ({{s}}))\, f(X ({{s}}))\rk. \lk.\,\mid \CY_t\rk]dY^c_s
 \nonumber
 \\&=&
 \int_0^ t  \EE^ {\QQ} \lk[ V ({{s}})g(X ({{s}}))\, f(X ({{s}}))\,\mid \CY_{{s}}\rk] dY^c_s.
\EEQSZ
\del{ Due to the fact that $f'(X(s^-))$ is a  $\mathcal{F}_{{s^-}}$--measurable random variable and $\Delta L_0=L_0(s)-L_0(s^-)$ is independent from $\mathcal{F}_{{s^-}}$,
again it follows  by Corollary 2.1, we get
\DEQSZ
\EE^ {\QQ}\big[ \int_0^ t  V ({{s^-}})f'(X(s^-))\, dL_0(s)| \mathcal{Y}_t\big]=\int_0^ t \EE^ {\QQ}\big[ V ({{s^-}})f'(X(s^-))| \mathcal{Y}_t\big]\, dL_0(s)
\\&=&\int_0^ t \EE^ {\QQ}\big[ V ({{s^-}})f'(X(s^-))| \mathcal{Y}_s\big]\, dL_0(s)
\EEQSZ
}
 Due to the fact that $V({{s^-}})\int_{\RR_0^+} \lk[ f( X ({s ^-} )+ y) -f( X ({s ^-} ) )\rk]\, \nu_{1,z_2}(dy)$ is a  $\CF_{{s^-}}$--measurable random variable and $\Delta L_2=L_2(s)-L_2(s^-)$ is independent from $\CF_{{s^-}}$,
 it follows again by Corollary \ref{cor23} % we obtain
 \DEQS
\lqq{ \EE^{\QQ}\lk[ \int_0^t \int_{\RR_0^+} V ({{s^-}})\,
 \int_{\RR_0^+} \lk[ f( X ({s ^-}) + z_1) -f( X ({s ^-} ))\rk]  \, \nu_{1,z_2} (dz_1)\eta_2(dz_2,ds) \mid \CY_t\rk]}
  \\
\lqq{= \int_0^t \int_{\RR_0^+}  \EE^ {\QQ}\lk[ V({{s^-}})\,
  \rk.}
   \\&& \lk.
\int_{\RR_0^+} \lk[ f( X ({s ^-}) + z_1) -f( X ({s ^-} ))\rk]  \, \nu_{1,z_2} (dz_1)\mid \CY_t\rk]  \eta_2(dz_2,ds)
\\
\lqq{ =\int_0^t  \int_{\RR_0^+}\int_{\RR_0^+} \EE^ {\QQ}\lk[  V({{s^-}})\,
  \rk.} \\&& \lk.
   \lk[ f( X ({s ^-}) + z_1) -f( X ({s ^-} ))\rk]  \, \nu_{1,z_2}(dz_1) \mid \CY_{{s^-}}\rk] \eta_2(dz_2,ds).
 \EEQS
  %$\CY_t$ and a martingale over $(\Omega,\CF,\QQ)$
 %we can write
 %
By collecting all the results, one can conclude the theorem.
\end{proof}
In the case where the \levy measure of $L$ is $\sigma$-finite,
the copula  has to satisfy certain scaling properties. Namely, we have to take $H$ such that %there exists a constant $C>0$ and
\DEQSZ\label{scaling}
\lim_{\gamma\to\infty} { H(\gamma u,\gamma v )\over H(\gamma,\gamma)} =  H(u,v),\quad u,v\in\RR.
\EEQSZ

Now we can formulate the following Theorem for the case where \levy measure of $L$ is $\sigma$-finite.
\begin{thm}\label{copula_inf}
Let $L_0$ be a \levy process with symbol $\psi$.
Let $\nu_1$ and $\nu_2$ be two $\sigma$-finite \levy measures  such that
\DEQSZ\label{asslevym}
\int_{|z|\le 1}|z|\,\nu_1(dz)+\int_{|z|\le 1}|z|\,\nu_2(dz)<\infty
.\EEQSZ

Let $H$ be a twice differentiable copula which satisfies the scaling property \eqref{scaling}.
Let
\DEQS
\nu_{1,{z_2}}(U) &=& \int_U h(z_1,z_2)\, \nu_1(dz_1),
\quad U\in\CBB (\RR\setminus \{ 0\}),
\EEQS
where
$$
h(z_1,z_2):=
{ {\partial ^ 2 \over \partial u_1\partial u_2} H(u_1,u_2) \Big|_{u_1=U_1 (z_1)\atop u_2=U_2 (z_2)}
 },
$$
and $U_1$, $U_2$ are the tail integrals of $\nu_1$ and $\nu_2$, respectively.
Let $g,b:\RR\to\RR$ and $\sigma:\RR\to\RR$ are Lipschitz continuous mappings and $g\in C^{(2)}_b(\RR)$.
%For a given function $f\in C ^{(2)}_b(\RR)$ t
The un-normalized conditional density $\rho$ is a unique solution to the equation,
\DEQSZ\label{variational1}
\\
\nonumber
 \la \rho_t,f\ra  &=& \la \rho_0,f\ra +  \int_0^ t  \la
\rho_{{s^-}}\, ,f\cdot g \ra \, dY_s^ c
\\\nonumber  && {}+ \int_0^t \la \rho_{{s^-}},\CA_0  f\ra\, ds
%+    \sum_{s\le t} \la \rho_s,f\ra\, \phi(\Delta Y_{{s^-}})
%\\&&
 + \int_0^ t\int_\RR   %\int_\RR
 \la \rho_{{s^-}}, \OPER _{z_2}  f\, \ra\, \eta_2(dz_2,ds) %\nu_{1,\Delta L_2(s)}(dz_1)\,
, \quad \forall f\in C ^{(2)}_b(\RR),\EEQSZ
where $\eta_2$ is the Poisson random measure associated to $L_2$ and
%$\Delta L_2(t)$ denotes the jumps of $L_2$, i.e.\ $\Delta L_2(t)=\lim_{s\uparrow t} L_2(t)-L_2(s)$, %for $z_2\in\RR$
the operators $\OPER=\{\OPER_z: z\in\RR\setminus\{0\}\}$, $\CA_0$ are given  by
$$\OPER_{z} f(x)=   \int_\RR
\lk[f(x+z_1)-f(x)\rk] \nu_{1,z}(dz_1),\quad x\in\RR,
\,\,z\in\RR\setminus\{0\},\,\,
$$
 $$
\CA_0 f (x)= b(x)f'(x) + A_0 f(x),  \quad x\in \mathbb{R},\quad f\in C ^{(2)}_b(\RR).
%,\,\,f\in C ^{(2)}_b(\RR).
$$
Here $A_0$ is the pseudo--differential operator associated with $L_0$.
%and
%%Moreover, the solution is unique.
%$$
%A_0\, f = - \int_\RR e^{i \xi x } \psi(\xi) \hat f (\xi)\, d\xi, \quad f\in C^{(2)}_b(\RR).
%%A_0 f = \int_\RR \lk( f(x+y)-f(x)-f'(x)y\rk) \nu(dy), \quad f\in C ^{(2)}_b(\RR).
%$$
%
\end{thm}

\begin{rem}\label{xiinverse}
\del{We already mentioned that
$$
\xi(t) \pi_t(f)=\rho_t(f),\quad t\ge 0,
$$ and
$$
 \pi_t(f)=\rho_t(f)\xi^ {-1}(t),\quad t\ge 0.
$$
where $\xi$ is defined in $\eqref{xiintro}$.}
\del{\DEQSZ\label{xiintro}
\xi(t)=\EE^\QQ\lk[ V(t)\mid \CY_t\rk],\quad t\ge 0.
\EEQSZ}
By taking $f=1$ in \eqref{variational1} and taking into account that $\CA_0 1=0$, $\OPER _z1=0$, it follows that  $\xi$ solves
\DEQS
\xi(t) &=& 1+\int_0^t \rho_{s}(g) \, dY_s^c=1+\int_0^t \rho_{s^-}(1) \, \pi_{{s}}(g) \, dY_s^c
\\
&=& 1+\int_0^t \xi(s) \, \pi_{{s}}(g) \, dY_s^c ,\quad t\ge 0.
\EEQS
Second and third equalities hold due to Kallianpur-Streibel formula and the fact that $\rho_{s}(1)=\xi(s)$ respectively.
Hence, the inverse $\varsigma=\{ \varsigma(t):t\ge 0\}$ of $\xi$ is given by
\DEQS
\varsigma (t) &=&\varsigma(0)+ \int_0^t \varsigma({s^-})\pi_s(g)^2\, ds -\int_0^t \pi_s(g) dY_s^c
\\ &=& \varsigma(0)+ \int_0^t \varsigma({s})^3\rho_s(g)^2\, ds -\int_0^t \varsigma({s})^2 \rho_s(g) dY_s^c
. \EEQS
Since $g\in C^{(2)}(\RR)$, one can easily show that $\rho(g)=\{
\rho_t(g):t\ge 0\}$ is bounded by $|
g|_{C_b}$ and is well defined. Due to this fact and the Novikov condition, we can see that the process $\varsigma$ exists and well defined.
\end{rem}

\begin{proof}
To start with the proof, firstly let us cut off the small jumps from the \levy process $L$.
For any $\ep>0$, let $\nu_1^\ep=\nu_1(\cdot \cap
\RR\setminus(-\ep,\ep))$, $\nu_2^\ep=\nu_2(\cdot
\cap \RR\setminus(-\ep,\ep))$, and
$\lambda_1^\ep=\nu_1^\ep(\RR)$, $\lambda_2^\ep=\nu_2^\ep(\RR)$. We denote by $L_1^\ep$ and $L_2^
\ep$ the \levy processes corresponding to the \levy measures
$\nu_1^\ep$ and $\nu_2^\ep$, respectively. As before, $\QQ_\eps$
be a probability measure such that \DEQS { d \PP\over
d\QQ_\eps}\Big|_{\CF_t}&=& V^\eps (t),\quad t\ge 0, \EEQS where
$X^\ep$ solves
\DEQSZ\label{eqn-xep}
\lk\{ \barray  dX^ \ep(t) &=& b(X^ \ep(t))\, dt + dL_0(t)+ d L^
\ep_1(t),\quad t>0,
\\ X^\ep(0)&=& X^\ep_0\earray\rk.
\EEQSZ
and $V^\ep$ solves
 %solves
\DEQSZ\label{eqn-vep} \lk\{\barray dV^ \ep(t) &{=}& V^ \ep ({{t}}) g(X^
\ep({{t}}))\,\lk[  dW_2(t) +g(X^ \ep(t))\, dt \rk]
\\ V^ \ep(0)&=&1.
\earray\rk. \EEQSZ
 Let $\rho^ \ep=\{ \rho_t^\ep:t\ge 0\}$ be the un-normalized conditional
density process given by
$$
\rho_t^ \ep(f)=\EE ^ {\QQ_\ep}  \lk[ V^ \ep(t)\, f(X^
\ep(t))\mid\CY_t^\ep\rk] ,
$$
%where $X^ \ep=\{ X^ \ep(t):t\ge 0\}$ solves
%
%Let
and $Y^ \ep=\{ Y^ \ep(t):t\ge 0\}$ be the solution to %the SDE%solves %be given by the  following SDE
 \DEQSZ\label{eqn-yep}
\lk\{ \barray  dY^ \ep(t) &=& %\int_0^ t
g(X^ \ep(t) )\, dt + dL^ \ep_2(t)+ dW_2(t) ,\quad t>0,\\ Y^
\ep(0)&=& Y^\ep_0.\earray\rk. \EEQSZ
Notice that under the probability measure $\QQ_\ep$, the
continuous part of $Y^ \ep$ is a Brownian motion.

Let us denote the number of jumps of $L_2^ \ep$ in the time interval
$[0,t]$ by $N_\ep(t)$, the jumps themselves by $\{ Y_
{2,\ep,i}:i=1,\ldots, N_\ep(t)\}$, and the jump times by $\{
s^\ep_i: i=1,\ldots,{N_\ep(t)}\}$.
Then, %given the jumps of $L^\ep_2$ in the time interval $[0,t]$, $L^\ep_1$ is given by
$$
L^\ep_1(t) = \sum_{i=1}^ {N_\ep(t)} Y_{Y_
{2,\ep,i}}^ {1,\ep,i},\quad t\ge 0,
$$
where $\{ Y^ {1,\ep,i}_{Y_
{2,\ep,i}} :i=1,\ldots, N_\ep(t)\}$ is a family of independent random variables. For any $i=1,\ldots,N_\ep(t)$,
the random variable $Y^ {1,\ep,i}_{Y_
{2,\ep,i}}$ is distributed by $\nu^\ep_{1,z}/\lambda_1^\ep$ with $z=Y_
{2,\ep,i}$.
Now following the same calculations as in the proof of Theorem
\ref{copula_finite}, we get
\DEQS f(X^ \ep (t)) &=& f(X^ \ep _0) +
\int_0 ^t\lk( \CA_0\,f\rk) (X^ \ep (s))\, ds+M_\ep(t)
\\ && {} +  \sum_{1\le i\le N_\ep (t)}  f( X^ \ep ({s_i ^-}) +  Y_{Y_
{2,\ep,i}}^ {1,\ep,i}) -f( X^ \ep ({s_i ^-}) ),
\EEQS
where $M_\ep$ is a martingale and $\EE^\QQ\lk[ M_\ep(t)\mid \CY_t\rk]=0$.
Put
\DEQS
\nu^ \ep_{1,{z_2}}(U) &=& \int_{U\cap \lk[ (-\infty,-\ep]\cup [\ep,\infty)\rk]} h(z_1,z_2)\, \nu_1(dz_1),
\quad U\in\CBB (\RR).
\EEQS
Similarly as in Theorem 2.2, we denote the Poisson random measure corresponding to $L_2^\ep$ by $\eta_2^\ep$.
Thus, we can write
\DEQS
\lqq{ f(X^ \ep (t)) = f(X ^ \ep_0) + \int_0 ^t \lk( \CA_0 f\rk) (X ^ \ep(s))\,  ds+M_\ep(t)} &&
\\ && +  \sum_{1\le i\le N_\ep(t)}  f( X ^ \ep({s_i ^-}) + Y_{Y_
{2,\ep,i}}^ {1,\ep,i}) -f( X^ \ep ({s_i ^-}) )
\\ &=&  f(X ^ \ep_0) + \int_0 ^t \lk( \CA_0 f\rk) (X ^ \ep(s))\,  ds+M_\ep(t)
\\ && +  \sum_{1\le i\le N_\ep(t)}  f( X ^ \ep({s_i ^-}) + Y_{Y_
{2,\ep,i}}^ {1,\ep,i}) -f( X^ \ep ({s_i ^-}) )
\\
&{} - & \int_0^t  \int_\RR  \int_\RR \lk[ f( X^ \ep({s_i ^-})  + z_1) -f( X^ \ep ({s_i ^-}) )\rk] \, \nu^ \ep_{1, z_2 } (dz_1)\, \eta_2^\ep(dz_2,ds)
\\
& {} + &\int_0^t  \int_\RR  \int_\RR \lk[ f( X^ \ep({s_i ^-})  + z_1) -f( X^ \ep ({s_i ^-}) )\rk] \, \nu^ \ep_{1, z_2 } (dz_1)\, \eta_2^\ep(dz_2,ds)
\\
&=& f(X^ \ep _0) +  \int_0 ^t \CA_0 f(X^ \ep (s))\, ds+ M_\ep (t)+\tilde M _\ep (t)
\\ &&{}+\int_0^t  \int_\RR  \int_\RR \lk[ f( X^ \ep({s_i ^-})  + z_1) -f( X^ \ep ({s_i ^-}) )\rk] \, \nu^ \ep_{1, z_2 } (dz_1)\, \eta_2^\ep(dz_2,ds)
,
\EEQS
 By using same arguments in the proof of Theorem \eqref{copula_finite}, we can show that  for $t\ge 0$ we have $\EE^{\QQ_\ep}[\tilde M _\ep(t)|\CY_t^\ep]=0$.
 \del{and $\tilde{M}  _\ep=\{ \tilde{M}_\ep(t):t\ge 0\}$ is also a martingale on $
(\Omega,\CF,\CF_t,\QQ_\ep)$}
Next, the process $V^  \ep=\{ V^  \ep (t):t\ge 0\}$ satisfies under $\QQ_\ep$ the stochastic differential equation
$$
dV^  \ep (t) = V^  \ep (t)\,g(X^\ep ({{t}}))\,  dY^ {c} (t),\, \quad t>0, \quad V^  \ep (0)=1,
$$
where $Y^ {c} $ denotes the continuous part of $Y^ \ep $ and it does not depend up on $\ep$.
Since $V^ { \ep} $ is driven by the continuous part of $Y^ \ep $ and the jumps times are given, there will be no correlation terms in the formula for $V^ { \ep}(t)$.
Thus, we get
\DEQS
\lqq{
%\EE ^ \QQ[ f(X^  \ep(t))] =\EE^ \PP\lk[
 f(X^  \ep (t))\, V^  \ep (t)=  f(X^  \ep_0)+ \int_0^ t  V^  \ep ({{s^-}})\,dM _\ep(s)+ \int_0^ t  V^  \ep ({{s^-}})\,d\tilde{M} _\ep(s)
 %\mid \CY(t) \rk]
  }
&&
\\
&+& %\int_0^t f(X^  \ep_s)\, dV^  \ep_s +
  \int_0^t   \int_\RR \int_\RR V^  \ep ({{s^-}})
 \lk[ f( X^  \ep ({s ^-}) + z_1) -f( X^  \ep ({s ^-} ))\rk]  \, \nu^ \ep_{1,z_2} (dz_1)\eta_2^\ep(dz_2,ds)
 %   \sum_{1\le i\le N_  \ep (t)} V^  \ep ({{s^-}})\, \int_\RR \lk[ f( X^  \ep ({s_i ^-}) + y) -f( X^  \ep ({s_i ^-}) )- f'(X^  \ep ({s_i ^-} ))y\rk] \, \nu^ \ep_{1, Y_i^{2,\ep}} (dy)/\lambda_\ep
%\underbrace
\\ &+&{}
\int_0^ t V^  \ep ({{s}})\, f(X^  \ep ({{s}}))\, dY^ {c}(s)
 + \int_0^ t V^  \ep ({{s}}) \lk( \CA_0 f\rk) (X^  \ep ({{s}})) \,
 ds
 .
\EEQS
Note that $$\CY_t^\ep=\sigma\{Y_r:0\le r \le t, \ep\le|\Delta L_2(r)|<\infty\}.$$
Taking into account that $M_\ep$ is a martingale over  $(\Omega,\CF,(\CY_t^\ep)_{t\ge 0},\QQ_\ep)$, the fact that  $\EE^{\QQ_\ep}[\tilde M _\ep(t)|\CY_t^\ep]=0$ and taking the conditional expectation together with the Fubini Theorem \cite[Theorem 1.1.8]{applebaum} we get %to the above result,
\DEQS
\lqq{
 \EE^ {\QQ_\ep} \lk[ f(X^  \ep (t)) V^  \ep (t)\mid \CY_t^\ep\rk] = \EE^ {\QQ_\ep} \lk[f(X^\ep_0)\mid \CY_0\rk]}
&&
\\
& &{}
 + \underbrace{\EE^ {\QQ_\ep} \lk[ \int_0^ t  V^  \ep ({{s^-}})\,dM_ \ep (s)\mid \CY_t^\ep\rk]}_{=0} +\underbrace{\EE^ {\QQ_\ep} \lk[ \int_0^ t  V^  \ep ({{s^-}})\,d\tilde{M}_ \ep (s)\mid \CY_t^\ep\rk]}_{=0}
  \\ && {}+
  \EE^ {\QQ_\ep} \lk[ \int_0^ t V^  \ep ({{s^-}})\, f(X^  \ep ({{s^-}}))\, dY^{c,\ep}(s)\mid \CY_t^\ep\rk]+\EE^ {\QQ_\ep} \lk[  \int_0^ t V^  \ep ({{s^-}})  \, \lk( \CA_0 f\rk) (X^  \ep ({{s^-}}))\, ds\mid \CY_t^\ep\rk]
  \\\lqq{ %&&{}
\hspace{-0.5cm}
{}+ \EE^ {\QQ_\ep}\lk[ \int_0^t   \int_\RR \int_\RR V^  \ep ({{s^-}})\,
  \lk[ f( X^  \ep ({s ^-}) + z_1) -f( X^  \ep ({s ^-} ))\rk]  \, \nu^ \ep_{1,z_2} (dz_1)\eta_2^\ep(dz_2,ds) \mid \CY_t^\ep\rk]
 .} &&
\del{\\
& =&
 f(X^  \ep_0) +   \int_0^ t\EE^ {\QQ_\ep} \lk[ V^  \ep ({{s^-}} ) \lk( \CA_0 f\rk) (X^  \ep ({{s^-}}))\mid \CY_t^\ep\rk]\, ds
 \\\lqq{
+  \int_\RR \EE^ {\QQ_\ep}\lk[  V^  \ep ({s_i-})\, }
\\
\lqq{ \lk[ f( X^  \ep ({s_i ^-} )+ y) -f( X^  \ep ({s_i ^-}) )- f'(X^  \ep ({s_i ^-}) )y\rk] \mid \CY_t^\ep\rk] \, \nu^ \ep_{1,Y_i^{2,\ep}} (dy)/\lambda_\ep}
 &&
\\
&&{}+ \EE^ {\QQ_\ep} \lk[ \int_0^ t V^  \ep ({{s}})\, f(X^  \ep ({{s}}))\, dY(s)^ c\mid \CY_t^\ep\rk]
}\EEQS
By imitating the calculation \eqref{newmar} and  \eqref{mart} for $\int_0^ t  V^ \ep ({{s^-}})\,d\tilde{M}_ \ep (s)=\int_0^ t  V^ \ep ({{s^-}})\,dJ^ \ep (s)-\int_0^ t  V^ \ep ({{s^-}})\,dR^ \ep (s)$, we can again  show that
$$\mathrm{E}^ {\QQ_\ep}\big[ \int_0^ t  V ^ \ep({{s^-}})\,d\tilde{M }_\ep(s)| \mathcal{Y}_t^\ep\big]=0.
$$
Next by following the same calculations done in $\eqref{zero1}$ and $\eqref{zero2}$, we can  prove that
$$\mathrm{E}^ {\QQ_\ep}\big[ \int_0^ t  V ^ \ep({{s^-}})\,dM_\ep(s)| \mathcal{Y}_t^\ep\big]=0.
$$
 Now,  since $ X^  \ep ({{s}})$ is  $\CF_{{s}}$--measurable  we have
 \DEQS
 \EE ^ {\QQ_\ep} \lk[ V^  \ep ({{s}}) \,\lk( \CA_0 f\rk) (X^  \ep ({{s}}))\mid \CY_t^\ep\rk]
 &
 {=}&\EE ^{\QQ_\ep} \lk[ V^  \ep ({{s}}) \,\lk( \CA_0 f\rk) (X^  \ep ({{s}}))\mid \CY_{{s}}^\ep\rk].
\EEQS
 Note that since $Y^ {c}(t)$ is $\CY_t^\ep$-measurable, we have similarly as in Theorem 2.2,
 \DEQS
 \EE^ {\QQ_\ep}\lk[ \int_0^ t V^  \ep ({{s^-}})g(X^  \ep ({{s^-}}))\, f(X^  \ep ({{s^-}}))\, dY_s^ {c}\mid \CY_t^\ep\rk]=
 \int_0^ t  \EE^ {\QQ_\ep} \lk[ V^  \ep ({{s^-}})g(X^  \ep ({{s^-}}))\, f(X^  \ep ({{s^-}}))\,\mid \CY_{{s^-}}^\ep\rk] dY_s^ {c}.
\EEQS
\del{
 Due to the fact that $f'(X(s^-))$ is a  $\mathcal{F}_{{s^-}}$--measurable random variable and $\Delta L_0=L_0(s)-L_0(s^-)$ is independent from $\mathcal{F}_{{s^-}}$,
again it follows  by Corollary 2.1, we get
\begin{equation}
\begin{split}
\EE^ {\QQ}\big[ \int_0^ t  V^  \ep ({{s^-}})f'(X^  \ep(s^-))\, dL_0(s)| \mathcal{Y}_t\big]=\int_0^ t \EE^ {\QQ}\big[ V^  \ep ({{s^-}})f'(X^  \ep(s^-))| \mathcal{Y}_t\big]\, dL_0(s)
\\=\int_0^ t \EE^ {\QQ}\big[ V^  \ep ({{s^-}})f'(X^  \ep(s^-))| \mathcal{Y}_s\big]\, dL_0(s)
\end{split}
\end{equation}
}
 Since $ V^\ep ({{s^-}})\int_\RR\lk[ f( X^  \ep ({s_i ^-} )+ y) -f( X^  \ep ({s_i ^-} ) )\rk]\, \nu^ \ep_{1,z_2}(dy)$ is an
 $\CF_{{s^-}}$--measurable random variable, $L_2$ is a \levy process with respect to $(\Omega,\CY,(\CY_t^\ep)_{t\ge 0},\QQ^\ep)$ , we obtain
\DEQS
\lqq{ \EE^ {\QQ_\ep}\lk[ \int_0^t   \int_\RR \int_\RR   V^  \ep ({{s^-}})\,
 \lk[ f( X^  \ep ({s ^-}) + z_1) -f( X^  \ep ({s ^-} ))\rk]  \, \nu^ \ep_{1,z_2} (dz_1)\eta_2^\ep(dz_2,ds) \mid \CY_t^\ep\rk]}
  \\ &=&
\int_0^t   \int_\RR \EE^ {\QQ_\ep}\lk[ \int_\RR  V^  \ep ({{s^-}})\,
\lk[ f( X^  \ep ({s ^-}) + z_1) -f( X^  \ep ({s ^-} ))\rk]  \, \nu^ \ep_{1,z_2} (dz_1)\mid \CY_t^\ep\rk] \eta_2^\ep(dz_2,ds)
 \\ &=&
\int_0^t   \int_\RR  \int_\RR   \EE^ {\QQ_\ep}\lk[ V^  \ep ({{s^-}})\,
 \lk[ f( X^  \ep ({s ^-}) + z_1) -f( X^  \ep ({s ^-} ))\rk] \mid \CY_{{s^-}}^\ep \rk] \, \nu^ \ep_{1,z_2} (dz_1) \,\eta_2^\ep(dz_2,ds)
.  \EEQS
Now collecting all the terms, we get
\DEQSZ \label{epzakai}\lqq{ \EE^ {\QQ_\ep} \lk[ f(X^  \ep (t)) V^  \ep (t)\mid \CY_t^\ep\rk]
  {=}
 \EE^ {\QQ_\ep} \lk[f(X^\ep_0)\mid \CY_0\rk] + \int_0^ t\EE^ {\QQ_\ep} \lk[ \,V^  \ep ({s})\,\lk( \CA_0 f\rk) (X^  \ep ({{s}}))\mid \CY_{{s}}^\ep\rk]\, ds
}&&  \\
\nonumber &&{}+ \int_0^ t  \EE^ {\QQ_\ep} \lk[ g(X^  \ep ({{s}}))\, f(X^  \ep ({{s}}))\,\mid \CY_{{s}}^\ep\rk] dY^ {c}(s)
 \\
\nonumber&&{}+\int_0^t   \int_\RR  \int_\RR   \EE^ {\QQ_\ep}  \lk[ V^  \ep ({{s^-}})\lk[ f( X^  \ep ({s ^-}) + z_1) -f( X^  \ep ({s ^-} ))\rk] \mid \CY_{{s^-}}^\ep \rk] \, \nu^ \ep_{1,z_2} (dz_1)\, \eta_2^\ep(dz_2,ds)
.
 \EEQSZ

Now we would like to pass to the limit %$\ep\rightarrow 0$
 and to get the desired Zakai equation. By \cite[p.\ 235
Corollary  4.3.10 and p.\ 392, Theorem 6.5.2]{applebaum} it follows  $ X^
\ep\to X$ and $Y^ \ep\to Y$ uniformly on compact interval almost surely. Hence, the term $ \EE^ {\QQ_\ep} \lk[f(X^\ep_0)\mid \CY_0\rk]$ converges to $\EE^ {\QQ} \lk[f(X_0)\mid \CY_0\rk]$ as $\ep\to0$.
Fix $t\ge 0$. Because of the above fact, we  apply Theorem \ref{mainC} to show that for any $s\in[0,t]$, $\QQ$--a.s.\
$$\EE^ {\QQ_\ep} \lk[  V^  \ep (s)(\CA_0 f)(X^  \ep (s))\mid \CY_s^\ep\rk] \to \EE^ {\QQ} \lk[ V (s) (\CA_0 f)(X (s)) \mid \CY_s\rk], \quad \ep\to 0.
$$
The Lebesgue dominated convergence Theorem gives that $\int_0^ t\EE^ {\QQ_\ep} \lk[ \,V^  \ep ({s})\,\lk( \CA_0 f\rk) (X^  \ep ({{s}}))\mid \CY_{{s}}^\ep\rk]\, ds$
converges to $\int_0^ t\EE^ {\QQ} \lk[ \,V ({s})\,\lk( \CA_0 f\rk) (X ({{s}}))\mid \CY_{{s}}\rk]\, ds$.
Next, again applying  Theorem \ref{mainC} gives for any $s\in[0,t]$ that
$ \EE^ {\QQ_\ep} \lk[ g(X^  \ep ({{s}}))\, f(X^  \ep ({{s}}))\,\mid \CY_{{s}}^\ep\rk]$ converges to $ \EE^ {\QQ} \lk[ g(X ({{s}}))\, f(X({{s}}))\,\mid \CY_{{s}}\rk]$.
Again the Burkholder-Gundy-Davis inequality and the Lebesgue dominated convergence Theorem gives
that   $$
   \int_0^ t  \EE^ {\QQ_\ep} \lk[ g(X^  \ep ({{s}}))\, f(X^  \ep ({{s}}))\,\mid \CY_{{s}}^\ep\rk] dY^ {c}(s)
   $$
converges  to
$$\int_0^ t  \EE^ {\QQ} \lk[ g(X ({{s}}))\, f(X({{s}}))\,\mid \CY_{{s}}\rk] dY^ {c}(s)
$$
as $\ep\rightarrow 0$.
Our final goal is to prove that
$$
\int_0^t   \int_\RR  \int_\RR   \EE^ {\QQ_\ep}  \lk[ V^  \ep ({{s^-}})\lk[ f( X^  \ep ({s ^-}) + z_1) -f( X^  \ep ({s ^-} ))\rk] \mid \CY_{{s^-}}^\ep \rk] \, \nu^ \ep_{1,z_2} (dz_1)\, \eta_2^\ep(dz_2,ds)
$$
converges
 to
 $$\int_0^t   \int_\RR  \int_\RR   \EE^ {\QQ}  \lk[ V ({{s^-}})\lk[ f( X ({s ^-}) + z_1) -f( X ({s ^-} ))\rk] \mid \CY_{{s^-}} \rk] \, \nu_{1,z_2} (dz_1)\, \eta_2(dz_2,ds)$$ as $\ep\rightarrow 0$. For the notational convenient, we use
 $$\mathscr{U}^\ep_{t,z_1,z_2}=
 %\int_0^t   \int_\RR  \int_\RR
 \EE^ {\QQ_\ep}  \lk[ V^  \ep ({{t^-}})\lk[ f( X^  \ep ({t ^-}) + z_1) -f( X^  \ep ({t ^-} ))\rk] \mid \CY_{{t^-}}^\ep \rk] %\, \nu^ \ep_{1,z_2} (dz_1)\, \eta_2^\ep(dz_2,dt)
 $$
 and
 $$\mathscr{U}_{t,z_1,z_2}=
 %\int_0^t   \int_\RR  \int_\RR
 \EE^ {\QQ}  \lk[ V ({{t^-}})\lk[ f( X ({t ^-}) + z_1) -f( X ({t ^-} ))\rk] \mid \CY_{{t^-}} \rk]. % \, \nu_{1,z_2} (dz_1)\, \eta_2(dz_2,dt).
$$
 Now consider
\DEQSZ\label{lasttermofzakai}
&&  \EE^ {\QQ}  \lk| \int_0^t   \int_\RR  \int_\RR \lk[\mathscr{U}^\ep_{s,z_1,z_2}  \, \nu^ \ep_{1,z_2} (dz_1)\, \eta_2^\ep(dz_2,ds)- \mathscr{U}_{s,z_1,z_2}  \, \nu_{1,z_2} (dz_1)\, \eta_2(dz_2,ds) \rk]\rk|
\\
&\le&  \EE^ {\QQ}  \lk| \int_0^t   \int_\RR  \int_\RR \lk[\mathscr{U}^\ep_{s,z_1,z_2}\mathbf{1}_{(-\ep,\ep)^c}(z_1)-  \mathscr{U}_{s,z_1,z_2}\rk]h(z_1,z_2)\nu_1 (dz_1)\, \eta_2^\ep(dz_2,ds) \rk|
\nonumber
\\
&+&  \EE^ {\QQ}  \lk| \int_0^t   \int_\RR  \int_\RR \mathscr{U}_{s,z_1,z_2}h(z_1,z_2)\nu_1 (dz_1)\, \lk[\eta_2^\ep(dz_2,ds)-  \eta_2(dz_2,ds)\rk] \rk|
\nonumber
\EEQSZ
The first term in right hand side gives
\DEQSZ\label{lasttermofzakai1}
&&  \EE^ {\QQ}  \lk| \int_0^t   \int_\RR  \int_\RR \lk[\mathscr{U}^\ep_{s,z_1,z_2}\mathbf{1}_{(-\ep,\ep)^c}(z_1)-  \mathscr{U}_{s,z_1,z_2}\rk]h(z_1,z_2)\nu_1 (dz_1)\, \eta_2^\ep(dz_2,ds) \rk|
\\
&\le&  \EE^ {\QQ}  \lk| \int_0^t   \int_\RR  \int_\RR \lk[\mathscr{U}^\ep_{s,z_1,z_2}\mathbf{1}_{(-\ep,\ep)^c}(z_1)-  \mathscr{U}_{s,z_1,z_2}\rk]h(z_1,z_2)\nu_1 (dz_1)\, \tilde{\eta}_2^\ep(dz_2,ds) \rk|
\nonumber
\\
&+&  \EE^ {\QQ}  \lk| \int_0^t   \int_\RR  \int_\RR \lk[\mathscr{U}^\ep_{s,z_1,z_2}\mathbf{1}_{(-\ep,\ep)^c}(z_1)-  \mathscr{U}_{s,z_1,z_2}\rk]h(z_1,z_2)\nu_1 (dz_1)\mathbf{1}_{(-\ep,\ep)^c}(z_2)\nu_2 (dz_2)ds \rk|,
\nonumber
\EEQSZ
where $\tilde{\eta}_2^\ep(dz_2,ds)=\eta_2^\ep(dz_2,ds)-\nu_2^\ep (dz_2)ds$ and $(-\ep,\ep)^c=\RR\setminus(-\ep,\ep)$. Applying the Burkholder-Gundy-Davis inequality, H\"older inequality and Jensen's inequality
\DEQS
\ldots  &&
 \le C(t) \EE^ {\QQ}  \lk| \int_0^t   \int_\RR  \lk[\int_\RR \lk(\mathscr{U}^\ep_{s,z_1,z_2}\mathbf{1}_{(-\ep,\ep)^c}(z_1)-  \mathscr{U}_{s,z_1,z_2}\rk)h(z_1,z_2)\nu_1 (dz_1)\rk]^2\, \mathbf{1}_{(-\ep,\ep)^c}(z_2)\nu_2 (dz_2)ds \rk|^\frac12
\\
&&+\EE^ {\QQ}  \lk| \int_0^t   \int_\RR  \int_\RR \lk[\mathscr{U}^\ep_{s,z_1,z_2}\mathbf{1}_{(-\ep,\ep)^c}(z_1)-  \mathscr{U}_{s,z_1,z_2}\rk]h(z_1,z_2)\nu_1 (dz_1)\mathbf{1}_{(-\ep,\ep)^c}(z_2)\nu_2 (dz_2)ds \rk|
\\
&& \le  C(t)\lk(  \int_0^t   \int_\RR  \EE^ {\QQ} \lk| \int_\RR \lk(\mathscr{U}^\ep_{s,z_1,z_2}\mathbf{1}_{(-\ep,\ep)^c}(z_1)-  \mathscr{U}_{s,z_1,z_2}\rk)h(z_1,z_2)\nu_1 (dz_1)\rk|^2 \, \mathbf{1}_{(-\ep,\ep)^c}(z_2)\nu_2 (dz_2)ds \rk)^\frac12
\\
&&+   \int_0^t   \int_\RR \EE^ {\QQ}\lk| \int_\RR \lk[\mathscr{U}^\ep_{s,z_1,z_2}\mathbf{1}_{(-\ep,\ep)^c}(z_1)-  \mathscr{U}_{s,z_1,z_2}\rk]h(z_1,z_2)\nu_1 (dz_1)\rk|\mathbf{1}_{(-\ep,\ep)^c}(z_2)\nu_2 (dz_2)ds  .
\EEQS
\del{Due to Assumption \ref{asslevym}, $\int_\RR \lk(\mathscr{U}^\ep_{s,z_1,z_2}\mathbf{1}_{(-\ep,\ep)^c}(z_1)-  \mathscr{U}_{s,z_1,z_2}\rk)h(z_1,z_2)\nu_1 (dz_1)$ converges to $0$.
By simple arguments together with Theorem \ref{mainC}, we can show that the two terms in above inequality, i.e.}
Due to Assumption \ref{asslevym} and using simple arguments together with Theorem \ref{mainC} and Lebesgue Dominated Convergence theorem, we can show that the two terms in above inequality, i.e.\
$$\EE^ {\QQ}\lk|\int_\RR [\mathscr{U}^\ep_{s,z_1,z_2}\mathbf{1}_{(-\ep,\ep)^c}(z_1)-  \mathscr{U}_{s,z_1,z_2}]h(z_1,z_2)\nu_1 (dz_1)\rk|^2
$$
and
$$\EE^ {\QQ}\lk|\int_\RR [\mathscr{U}^\ep_{s,z_1,z_2}\mathbf{1}_{(-\ep,\ep)^c}(z_1)-  \mathscr{U}_{s,z_1,z_2}] h(z_1,z_2)\nu_1 (dz_1)\rk|
$$
 converge  to zero as $\ep\rightarrow 0$. Then by the Lebesgue Dominated Convergence theorem, the two terms in right hand side of above inequality converge to zero as  $\ep\rightarrow 0$. Let us consider the second term in the right hand side of \eqref{lasttermofzakai},
\DEQSZ\label{lasttermofzakai2}
&&  \EE^ {\QQ}  \lk| \int_0^t   \int_\RR  \int_\RR \mathscr{U}_{s,z_1,z_2}h(z_1,z_2)\nu_1 (dz_1)\, \lk[\eta_2^\ep(dz_2,ds)-  \eta_2(dz_2,ds)\rk] \rk|
\\
&\le&  \EE^ {\QQ}  \lk| \int_0^t   \int_\RR  \int_\RR \mathscr{U}_{s,z_1,z_2}h(z_1,z_2)\nu_1 (dz_1)\, \lk[\tilde{\eta}_2^\ep(dz_2,ds)-  \tilde{\eta}_2(dz_2,ds)\rk] \rk|
\nonumber
\\
&+&  \EE^ {\QQ}  \lk| \int_0^t   \int_\RR  \int_\RR \mathscr{U}_{s,z_1,z_2}h(z_1,z_2)\nu_1 (dz_1)\,\mathbf{1}_{(-\ep,\ep)}(z_2)\nu_2 (dz_2)ds \rk|.
\nonumber
\EEQSZ
The Burkholder-Gundy-Davis inequality and Jensen's inequality imply
\DEQS
\ldots  &&  \le  \EE^ {\QQ}  \lk| \int_0^t   \int_\RR  \lk|\int_\RR \mathscr{U}_{s,z_1,z_2}h(z_1,z_2)\nu_1 (dz_1)\rk|^2\, \mathbf{1}_{(-\ep,\ep)}(z_2)\nu_2 (dz_2)ds  \rk|^\frac12
\\
&+&  \EE^ {\QQ}  \lk| \int_0^t   \int_\RR  \int_\RR \mathscr{U}_{s,z_1,z_2}h(z_1,z_2)\nu_1 (dz_1)\,\mathbf{1}_{(-\ep,\ep)}(z_2)\nu_2 (dz_2)ds \rk|
\\
&&  \le    \lk[ \int_0^t   \EE^ {\QQ}\lk|  \int_\RR \int_\RR \mathscr{U}_{s,z_1,z_2}\nu_{1,z_2} (dz_1)\rk|^2\, \mathbf{1}_{(-\ep,\ep)}(z_2)\nu_2 (dz_2)ds  \rk]^\frac12
\\
&+&     \int_0^t   \int_\RR \EE^ {\QQ}\lk| \int_\RR \mathscr{U}_{s,z_1,z_2}\nu_{1,z_2} (dz_1)\rk|\,\mathbf{1}_{(-\ep,\ep)}(z_2)\nu_2 (dz_2)ds .
\EEQS
Again, arguing as before and using assumption \ref{asslevym}, we see that the two terms in right hand side of the above inequality go to zero as $\ep\rightarrow 0$.

\medskip
Summarizing, we have shown that for any $t\ge 0$, $ \EE^ {\QQ_\ep} \lk[ f(X^  \ep (t)) V^  \ep (t)\mid \CY_t^\ep\rk]$ converges to
$ \EE^ {\QQ} \lk[ f(X (t)) V (t)\mid \CY_t\rk]$ $\QQ$--a.s.. It is straightforward to see that the family of processes $[0,T]\ni t \mapsto  \EE^ {\QQ_\ep} \lk[ f(X^  \ep (t)) V^  \ep (t)\mid \CY_t^\ep\rk]$ is tight. Hence, we know by Theorem 7.8 by \cite{ethier}, that the process $ \EE^ {\QQ_\ep} \lk[ f(X^  \ep (\cdot)) V^  \ep (\cdot )\mid \CY_\cdot^\ep\rk]$
converges to the process $ \EE^ {\QQ} \lk[ f(X (\cdot)) V (\cdot)\mid \CY_\cdot\rk]$ in $\DD([0,T];\RR)$.

\del{ Then we can see that first term in the last inequality goes to zero as $\ep\rightarrow 0$ since $ X^
\ep\to X$ and $Y^ \ep\to Y$ uniformly on compact interval almost surely \cite[p.\ 235
Corollary  4.3.10 and p.\ 392, Theorem 6.5.2]{applebaum} and by using Lebesgue dominated convergence theorem. The second term in the last inequality approaches to zero as $\ep\rightarrow 0$
Using the definition of $\rho^ \ep _t$ %and the representation of the Poisson random measure {(see chapter in Tankov)}
we get
\DEQSZ\nonumber\label{oben}
\la \rho^ \ep _t,f\ra  &=& \la \rho^ \ep_0,f\ra +  \int_0^ t  \la \rho^ \ep _{{s^-}}, f.g\ra \,dY^ {c,\ep}(s)
\\ && {}+ \int_0^t \la \rho ^ \ep_{{s^-}}, \CA_0   f\ra\, ds
%+    \sum_{s\le t} \la \rho_s,f\ra\, \phi(\Delta Y_{{s^-}})
%\\&&
 + \int_0^ t\int_\RR  \la \rho ^ \ep_{{s^-}}, \OPER^ \ep_{z_2 } f\, \ra\, \eta^ \ep_{2}(dz_2,ds)
,
\EEQSZ
where
\DEQS
 \OPER  ^ \ep_{z_2} f(x)   =  \int_\RR
\lk[f(x+z_1)-f(x)\rk] \nu^ \ep  _{1,z_2}(dz_1),\quad f\in C ^{(2)}_b(\RR),\quad z_2\in\RR\setminus\{0\}.
\EEQS
\del{and $\nu^ \ep_2$ is a time homogenous Poisson random measure with \levy measure $\nu_2^\ep$ defined by
$\nu_2(\cdot\cap (\RR\setminus (-\ep,\ep)))$.
Hence,
\DEQSZ\nonumber\label{oben1}
\la \rho^ \ep_t,f\ra  &=& \la \rho^ \ep_0,f\ra +  \int_0^ t  \la \rho^ \ep_{{s^-}}, f\ra \,dY^ {\ep,c}(s)
\\ && {}+ \int_0^t \la \rho^ \ep_{{s^-}}, \CA_0   f\ra\, ds
%+    \sum_{s\le t} \la \rho_s,f\ra\, \phi(\Delta Y_{{s^-}})
%\\&&
 +\int_\RR  \la \rho^ \ep_{{s^-}}, \CBB ^ \ep_{z_2 } f\, \ra\, \eta_{2}^\ep(dz_2,dt)
.
\EEQSZ}
\del{where
$$\CBB ^ \ep_{z_2} f(x)=   \int_\RR
\lk[f(x+z_1)-f(x)\rk] \nu^ \ep _{1,z_2}(dz_1),\quad f\in C ^{(2)}_b(\RR).
$$}
Next, we have to take the limit $\ep\to0$. By \cite[p.\ 235
Corollary  4.3.10 and p.\ 392, Theorem 6.5.2]{applebaum} it follows  $ X^
\ep\to X$ uniformly on compact interval almost surely. Similarly
$Y^ \ep\to Y$ uniformly a.s. on compact intervals. It remains to
show that
$$
\OPER _z^ \ep f \to \OPER _z f
$$
uniformly a.s. on compact intervals. By \cite[p.\ 41, Theorem,
8.7]{sato} it remains to show that
$$
\lim_{\ep\to 0}  \int_{[0,\ep]} |\la z_2,x\ra |\nu_{1,z_2}(dx) \rightarrow 0.
$$
{By definition,
\begin{equation*}
\begin{split}
  \int_{[0,\ep]} \la z_2,x\ra\nu_{1,z_2}(dx)=\int_{[0,\ep]} |\la z_2,x\ra| h(x,z_2) \nu_{1} (dx).
   \end{split}
\end{equation*}
The above definition makes sense because $L_1$ is a finite variation \levy process, the second  derivative of $H$ is bounded for $x\neq \infty$
 and for each
$z_2$ the function $|\la z_2,x\ra |h(x,z_2)$ is in $L^1(\mathbb{R}, \nu_1)$. Therefore, the RHS of the above identity converges to $0$ when $\eps \to 0$ for any $z_2$.}
}
\del{Thanks to teh assumption on $f$ and $g$ it is easy to show that
\eqref{variational1} has a unique solution.}
\end{proof}

\section{Sufficient conditions for solvability of the Zakai equation}

In practice one is often interested in entities like
$$\PP( X(t)\ge a),\quad a\in\RR,
$$
 where $a$ is a given threshold. This correspond to the
case where $f=1_{[a,\infty)}$. Unfortunately, in this case
$f\notin C ^{(2)}(\RR)$ and we cannot expect that equation
\eqref{variational1} is well-posed. One method to handle this
problem is to treat equation \eqref{variational1} by the semigroup
approach. Let us denote the infinitesimal generator of the process
$L_0$ with the drift ( that is $ \int_0^.b(X(s))\, ds$) by $\CA_0$. If $\CA_0$ generates an analytic semigroup with  good smoothing property, then
one can show the existence of a measure valued solution to
\eqref{variational1} even for the case where $f=1_{[a,\infty)}$.
If the driving process $L_0$ of the state process $X$ is a Brownian motion,
then the operator $\CA_0$ in the Zakai equation \eqref{variational1} is the Laplace operator with first order operator. However, if $L_0$ is a \levy process of pure
 jump  type\footnote{We say that a \levy process is of pure jump type if it has no Gaussian part.}, then $\CA_0$ will be a pseudo differential operator.

There exists several approaches to deal with pseudo--operators arising from \levy processes. One way is
to define the operator $\CA_0$ associated with the symbol $\phi_{\CA_0}\footnote{If $\CA_0$ is the Laplacian, then $\phi_{\CA_0}(\xi)=\xi^ 2 $.}$ is given by
$$\phi_{\CA_0}(\xi) :=ib(x)\xi+ \int_\RR\lk( e^{i\xi z}-1\rk) \nu_0(dz),\quad \xi\in\RR.
$$
Here,  $L_0$ is a \levy process of pure jump type with intensity $\nu_0$.
For a short account on the associated symbol to a \levy process we refer to \cite{hsym}. More details can be found in the article of Hoh \cite{hoh1}, and in the books of Jacobs
\cite{Jacob-I,Jacob-II,Jacob-III}.

It can be shown that $\CA_0$ with domain $D(\CA_0)$ generates a
strongly continuous semigroup $T_{\CA_0}=(T_{\CA_0}(t) )_{t\ge 0}$
on $L^2(\RR^d)$. This semigroup can be extended (or restricted) to
a semigroup acting on $H^s_2(\RR^d )$, $s\in\RR$. By analyzing the
symbol $\phi_{\CA_0}$, one gets information about the smoothing
properties of the semigroup $T_{\CA_0}=(T_{\CA_0}(t) )_{t\ge 0}$.

\begin{defn}[compare \cite{hsym}]
Let $L$ be a \levy process with symbol $\psi$ and  $\psi\in C^ k(\RR^d\setminus\{0\})$ for some  $k\in\NN_0$. Then
the Blumenthal--Getoor index of order $k$ is defined by
$$
\beta := \inf_{\lambda>0\atop |\alpha|\le k} \lk\{\lambda: \lim_{|\xi|\to\infty} {|\partial^ \alpha_\xi \psi(\xi)|\over |\xi|^{\lambda-|\alpha|} }=0\rk\}.
$$
Let
$$
\beta^+ := \inf_{\lambda>0\atop |\alpha|\le k} \lk\{\lambda: \limsup_{|\xi|\to\infty} {|\partial^ \alpha_\xi \psi(\xi)|\over |\xi|^{\lambda-|\alpha|} }=0\rk\},
$$
be the upper and
$$
\beta^- := \inf_{\lambda>0\atop |\alpha|\le k} \lk\{\lambda: \liminf_{|\xi|\to\infty} {|\partial^ \alpha_\xi  \psi(\xi)|\over |\xi|^{\lambda-|\alpha|} }=0\rk\},
$$
be the lower Blumenthal--Getoor index $\beta^-$ of order $k$.
Here $\alpha$ denotes a  multi-index. If $k=\infty$ then Blumenthal--Getoor index of infinity order is defined by
$$
\beta := \inf_{\lambda>0\atop \alpha \,\,\mbox{\tiny \rm is a muliindex} } \lk\{\lambda: \lim_{|\xi|\to\infty} {|\partial^ \alpha_\xi \psi(\xi)|\over |\xi|^{\lambda-|\alpha|} }=0\rk\}.
$$

\end{defn}

In many cases the  index  can be calculated directly from the symbol and is known. A sequence of examples of the generalized Blumenthal--Getoor index, like the symmetric $\alpha$--stable process,
 tempered $\alpha$--stable process, Meixner process and normal inverse Gaussian process are given in \cite{hsym}.

Depending on the lower index of $L_0$ and the marginal \levy measures $\nu_1$ and $\nu_2$ of the \levy process $L$,
one can prove that there exists a unique measure valued process $\pi=\{\pi_t:t\ge 0\}$ such that $$
\pi_t (f) = \EE \lk[ f(X(t))\mid \CY_t\rk],\quad f\in B_b(\RR).
$$

\begin{thm}\label{measurevalued}
%, $L=(L_1,L_2)$ a $d$--valued \levy process, intensity measure $\nu$, with the marginals $\nu_1$ and $\nu_2$.
Let us assume that
\begin{itemize}
\item $X_0$ has distribution function $F$, which has a $L^2 $--integrable density with respect to the Lebesgue measure;
  \item the symbol $\psi_0$ associated to $L_0$ has lower Blumenthal--Getoor   index $\alpha_0^->1$ of order two,
\item $g\in H^\delta_2(\RR)\cap C^{(2)}_b(\RR)$ with $\delta>1-\frac {\alpha_0^-}2$; %C^{(2)}_b(\RR)$;
  \item the symbol  $\phi_{\OPER _z}$ associated to the operator $\OPER _z$, has upper   Blumenthal--Getoor  index $\beta ^ +\le 1$ of order two,
  \item there exists some function $k:\RR^+_0\to\RR_0^+$, $k(0)=0$, continuous at $0$, such that
\DEQSZ\label{app1}
 \limsup_{|\xi|\to\infty} {|\phi_{\OPER _{z_2}}(\xi)|\over |\xi|^{\beta^ +} }\le  k(z_2) , \quad z_2\in\RR.
\EEQSZ
\item For simplicity, we take $L_1$ and $L_2$ with positive jumps such that $$\int_{|z_1|\le 1}|z_1|\nu_1(dz_1)+\int_{|z_2|\le 1}|z_2|\nu_2(dz_2)<\infty.$$
\end{itemize}
 In addition, if there exists a number $p\in (1,2]$ such that
\DEQSZ\label{app2} {\beta^ +\over \alpha^ -_0}<\frac 1p
 \quad \mbox{ and } \quad
\int_{|z_2|\le 1} \lk|k(z_2)\rk|^ { p}\nu_2(dz_2)<\infty,
\EEQSZ
then there exists a unique normalized conditional density
$\pi=\{ \pi_t:t\ge 0\}$ such that
$$
\pi_t (f) = \EE \lk[ f(X(t))\mid \CY_t\rk],\quad f\in B_b(\RR).
$$
Moreover for $f\in B_b(\RR)$,  $\pi_t(f)$ is given by
$$
\pi_t(f)  = \sigma(t)\cdot \rho_t(f)
$$
where $\sigma=\{\sigma(t):t\ge 0\}$ solves
\DEQS
\sigma(t)=1+\int_0^t \rho_{{s^-}}(g)\, dY^c_s,\quad t\ge 0,
\EEQS
and $\rho=\{\rho_t:t\ge 0\}$
is the unique solution of the following equation
\DEQSZ\label{eqn_mild1}
\lk\{ \barray d\rho_t &=&  \CA^\ast_0 \rho_t \, dt + \rho_t g  dY^ c_t+\int_\RR  \OPER^\ast _{z_2} \rho_{t^-}\eta_2(dz_2,dt),\\
\rho_0&=& \pi_0,
\earray \rk.
\EEQSZ
where
$\CA^\ast_0$ and $\OPER^\ast _{z_2}$ are adjoint operators of $\CA_0$ and $\OPER _{z_2}$.
\end{thm}

\begin{rm}
The adjoint operators $\CA^\ast_0$ and $\OPER^\ast _{z_2}$
are  defined as follows. Let $\sigma_{\CA_0}(x,\xi)$ be the symbol of the operator $\CA_0$. Therefore by using \cite[p.\ 26,
the adjoint operator representation (3.37)]{shubin}, the symbol of the operator $\CA^\ast_0$ can be read as
$$\sigma_{\CA^\ast_0}(x,\xi)\sim\sum_{|\alpha|\leq 1}\frac{\partial^{\alpha}_\xi D^{\alpha}_x  \overline{\sigma_{\CA_0}(x,\xi)}}{\alpha!}.$$ Similarly, the symbol $\phi^\ast_{z}(\xi)$ of the adjoint operator $\OPER^\ast _{z_2}$ is given by $$\phi^\ast_{z}(\xi)\sim\sum_{|\alpha|\leq s}\frac{\partial^{\alpha}_\xi   \overline{\phi_{z}(\xi)}}{\alpha!}.$$
\end{rm}
%\begin{rem}
%The restriction $\alpha_0^->1$ can be weakened.
%\end{rem}
\begin{proof}
%The key step for the existence of $\pi$ is a lemma of the existence of probability kernels by Getoor (see \cite[Proposition 4.1]{getoor}).
% Let $(\tilde \Omega,\CA)$ be a measurable space. Now, the lemma says, that if a operator $T:B_b(\RR,\CBB (\RR))\to B_b((\tilde \Omega,\CA))$
%\begin{itemize}
%  \item is linear a.e. and  positive a.e.;
%  \item for any sequence of functions $\{ f_n:n\in\NN\}$ and $f\in B_b(\RR)$,  $0\le f_n\uparrow f$ implies  $Tf_n\uparrow f$
%\end{itemize}
%%
%then there exists a bounded kernel $\mu(\cdot,\cdot)$ from $(\tilde \Omega,\CA)$ to $(\RR,\CBB (\RR))$ such that $Tf(\omega)=\int_\RR f(u)\, \mu(\omega,du)$, for all $f\in B_b(\RR)$ and $\omega\in\tilde \Omega$.
%Hence it is sufficient our framework satisfies the conditions of
We apply Theorem \ref{main} to get an $H^\frac 12_2 (\RR)$--valued solution, and then we show the existence of normalized conditional density by using the Getoor's lemma \cite[Proposition 4.1]{getoor} or \cite[Lemma 3.9]{Fer13}.

In fact if we take $\vr=\frac 12$, then one can easily see that
\eqref{app1} and \eqref{app2} imply that $A:=\CA^\ast$ and $G:=\OPER^\ast$
satisfy the assumptions of Theorem \ref{main}.
By \cite[Theorem
1 , p.\ 190]{runst} we have
$$
\lk|   u\, g \, \rk| _{H_2^ {\vr +\delta-\frac 12 } }\le \lk| u\rk| _{H_2^ {\vr } }\, \lk| g\rk| _{H_2^ {\delta} },\quad  u\in H_2^\vr(\RR) \quad \mbox{and} \quad g\in H_2^\delta(\RR).
$$
Therefore, by setting $\Sigma (u)=u\cdot g$ we also see that
$\Sigma$ satisfies the assumptions of Theorem \ref{main} as well.
Hence from these observations we see that if the assumptions of
Theorem \ref{measurevalued} hold, then it follows from Theorem
\ref{main} that there exists a $H^\frac 12_2 (\RR)$--valued
process $\rho$, such that for any $t\ge0$
$$
\rho_t(f) =\EE ^\QQ\lk[ V(t)\, f(X(t))\mid \CY_t\rk],\quad f\in\CBB (\RR).
$$
Secondly, let us fix $t>0$ and set $\CG=\CY_t$ as the
$\sigma$--field on $\Omega$ and define the operator $T$ by
$$
T f(\omega) = \EE \lk[ f(X(t))\mid \CY_t\rk](\omega).
$$
It is easy to check that $T$ is a.s.\ linear and positive.
Let $\{f_n:n\in\NN\}\subset B_b(\RR)$ be a sequence with $0\le f_n\uparrow f$. But if  $f_n\rightarrow f$ in $L^\infty(\RR)$, then one knows by Sobolev embedding theorem that $f_n\rightarrow f$ in $H^{-\frac 12}_2(\RR)$.
Since for $t>0$ $\rho_t$ is $H^\frac 12_2(\RR)$--valued un-normalized density measure,  $\rho_t(f_n)\rightarrow \rho_t(f)$. Here, one has  to take into account that
the density of $X_0$  belongs $\rho_0\in L^2 (\RR)$.
In addition, since $\rho_t(1)$ is well defined and invertible (see Remark \ref{xiinverse}), we have
$$
\pi_t(f_n) = {\rho_t(f_n)\over \rho_t(1)} \rightarrow  {\rho_t(f)\over \rho_t(1)}=\pi_t(f).
$$
That is $\pi_t(f_n)\to \pi_t(f)$. Since for $f_n\uparrow
f$, $f-f_n$ is a.s.\  positive, it follows that $T(f-f_n)$ is also
a.s.\ positive and, therefore, $Tf_n\uparrow Tf$.

Now, thanks to
 these two points we can infer from  \cite[Proposition 4.1]{getoor} or \cite[Lemma 3.9]{Fer13} that  there exists a kernel
$$\mu_t:(\Omega,\CY_t)\to  (\RR,\CBB (\RR)),$$
such that
$$\EE\lk[ f(X(t)) \mid\CY_t\rk](\omega) = \int_\RR f(u)\mu_t(\omega,du),\quad f\in B_b(\RR).
$$
\end{proof}

In the following corollary we present an example to illustrate the applicability of  Theorem \ref{measurevalued}.

\begin{cor}\label{spec_ex}
Let $L_0$ be a tempered $\alpha$--stable \levy process with $\alpha>1$ (see example 3.3)
with \levy measure
$$
\nu(U) =\int_U |z|^{-\alpha-1} e^ {-|z|} \, dz,\quad U\in\CBB (\RR),
$$
and $\nu_1$, $\nu_2$ are tempered   $\beta$--stable subordinators, $\beta\le 1$, %\levy processes
with \levy measure
$$
\nu(U) =\int_{U } |z|^{-\beta-1}  e^ {-|z|} dz,\quad U\in\CBB (\RR).
$$
Let  $g\in H^\delta_2(\RR)\cap C^{(2)}_b(\RR)\subseteq C^{(2)}_b(\RR)$ with $\delta>1-\frac {\alpha_0^-}2$.
Let $H$ be the Clayton copula with index $\theta>0$. % and $\beta=\frac 12$.
If  the distribution of $X_0$  has a $L^2 $ integrable density with respect to the Lebesgue measure,
then there exists a
unique family of probabilities kernels
$\pi=\{ \pi_t:t\ge 0\}$ such
$$
\pi_t (f) = \EE \lk[ f(X(t))\mid \CY_t\rk],\quad f\in B_b(\RR).
$$
Moreover for $f\in B(\RR)$ the kernel $\pi_t(f)$ is given by
$$
\pi_t(f)  = \sigma(t)\cdot \rho_t(f),
$$
where $\sigma=\{\sigma(t):t\ge 0\}$ solves
\DEQS
\sigma(t)=1+\int_0^t \rho_{{s^-}}(g)\, dY^c_s,\quad t\ge 0,
\EEQS
and $\rho=\{\rho_t:t\ge 0\}$
solves
\DEQSZ\label{eqn_mild11}
\lk\{ \barray d\rho_t &=&  \CA^\ast_0 \rho_t \, dt + \rho_t h  dY^ c_t +\int_\RR  \OPER^\ast _{z_2} \rho_{t^-}\eta_2(dz_2,dt),\\
\rho_0&=& \pi_0,
\earray \rk.
\EEQSZ
where $$\OPER _{z} f(x)=\int_\RR
\lk[f(x+z_1)-f(x) \rk] \nu_{1,z}(dz_1),\quad z\in\RR_+, x\in
\mathbb{R}, f\in C ^{(2)}_b(\RR).
$$
\end{cor}
\begin{proof}
By Theorem 1 of  \cite[p.\ 190]{runst} we have
$$
\lk|   u(r)\, g \, \rk|^ 2 _{H_2^ {\vr -\gamma/\alpha_0^-} }\le \lk| u(r)\rk| _{H_2^ {\vr } }\, \lk| g\rk| _{H_2^ {\delta} }.
$$

Now fix $z\in\RR\setminus\{0\}$. In the first step we will investigate the symbol $\phi_{\OPER _z}$ of $\OPER _z$. The operator $\OPER _z$ is reduced to following the form with
the Clayton copula for $f\in L^2 (\RR)$,
\DEQS
\lk(  \OPER_z \, f\rk)(x)
&=& \lk(1- \theta\rk) \int_0^ \infty\lk[ f(x+y)-f(x) \rk] \\ && \times \lk( |U_1(y)|^ {-\theta}+|U_2(z)|^ {-\theta} \rk) ^ {-\frac 1\theta-2} \, |U_1(y)|^ {-\theta-1}
|U_2(z)|^ {-\theta-1 }\, \nu_1(dy).
\EEQS
For us it is important to know the upper  index of
the symbol $\phi_{\OPER _z}$ associated to $\OPER_z$. The symbol $\phi_{\OPER _z}$ is given by
\DEQS
 \phi_{\OPER _z} \, (\xi)
 &=& \lk(1+\theta\rk) \int_0^ \infty  \lk[ e^{i\xi y} - 1   \rk]\,    \\ && \times \lk( |U_1(y)|^ {-\theta}+|U_2(z)|^ {-\theta} \rk) ^ {-\frac 1\theta-2} \, |U_1(y)|^ {-\theta-1}
|U_2(z)|^ {-\theta-1 }\, \nu_1(dy).
\EEQS
By the Clayton copula, we get
\DEQS
\phi_{\OPER _z}(\xi) &=&  \, (1+\theta)\,|U_2(z)|^ {-\theta-1 }\,
\\
&&{}\times \int_0^ \infty  \lk[ e^{i\xi y} - 1-i\xi y    \rk]\,
 \lk( |U_1(y)|^ {-\theta}+|U_2(z)|^ {-\theta} \rk) ^ {-\frac 1\theta-2} \, |U_1(y)|^ {-\theta-1}
 \nu_1(dy)
 \\&+&  \, (1+\theta)\,|U_2(z)|^ {-\theta-1 }\,
\\
&&{}\times \int_0^ \infty i\xi y \,
 \lk( |U_1(y)|^ {-\theta}+|U_2(z)|^ {-\theta} \rk) ^ {-\frac 1\theta-2} \, |U_1(y)|^ {-\theta-1}
 \nu_1(dy)
 \\&=& \, (1-\theta)\,|U_2(z)|^ {-\theta-1 }\,\times  2 (i\xi)^2  \int_0^\infty  \int_0^y \int_0^v
\\
&&{}  e^{i\xi u} \, du  \, dv  \,
 \lk( |U_1(y)|^ {-\theta}+|U_2(z)|^ {-\theta} \rk) ^ {-\frac 1\theta-2} \, |U_1(y)|^ {-\theta-1}
f_1(y)\, dy
\\&+&  \, (1+\theta)\,|U_2(z)|^ {-\theta-1 }\,
\\
&&{}\times \int_0^ \infty i\xi y \,
 \lk( |U_1(y)|^ {-\theta}+|U_2(z)|^ {-\theta} \rk) ^ {-\frac 1\theta-2} \, |U_1(y)|^ {-\theta-1}
 \nu_1(dy),
\EEQS
where $\nu_1(dy)=f_1(y)\, dy $.
One gets by
the Fubini's Theorem %gives  %for $\gamma>1$ %we can write
\DEQS
\lqq{ I_1(z) =  2 (1+\theta)\,|U_2(z)|^ {-\theta-1 }\,  (i\xi)^2 \lim_{R\rightarrow \infty}\int_{0}^R e^{i\xi y}
  }
 \\ &&{}  \lk[ \int_y^\infty \int_v^\infty \lk( |U_1(u)|^ {-\theta}+|U_2(z)|^ {-\theta} \rk) ^ {-\frac 1\theta-2} \, |U_1(u)|^ {-\theta-1}
 f_1(u)  \, du\, dv \rk] \, dy
  .
 \EEQS
Applying a version of Corput's Lemma (see \cite[p.\ 334 -
(6)]{stein}) we infer that
\DEQS
\lk| I_1  (z)\rk|  &\le &  2
(1+\theta)\,|U_2(z)|^ {-\theta-1 }\, |\xi|\,\lim_{R\rightarrow \infty}\lk|\int_{0}^R e^{i\xi y}\,dy\rk|
\\
&&
 \lk| \int_{0}^\infty
\int_v^\infty  \lk( |U_1(u)|^ {-\theta}+|U_2(z)|^ {-\theta} \rk) ^
{-\frac 1\theta-2} \, |U_1(u)|^ {-\theta-1}
 f_1(u)  \, \, du\, dv \rk|
 \\&\le&
  2
(1+\theta)\,|U_2(z)|^ {-\theta-1 }\, |\xi|\,
\\
&&
 \lk| \int_{0}^\infty
\int_v^\infty  \lk( |U_1(u)|^ {-\theta}+|U_2(z)|^ {-\theta} \rk) ^
{-\frac 1\theta-2} \, |U_1(u)|^ {-\theta-1}
 f_1(u)  \, \, du\, dv \rk|  .
\EEQS
Substitution with $m=U_1(u)$ gives the estimate
\DEQS
\lqq{ \lk| I_1
(z)\rk|\le 2 (1+\theta)\,|U_2(z)|^ {-\theta-1 }\, } &&\\
 & & |\xi|\, \lk|
\int_{0}^\infty \int_{U_1(v)}^0  \lk( |m|^ {-\theta}+|U_2(z)|^
{-\theta} \rk) ^ {-\frac 1\theta-2} \, |m|^ {-\theta-1} dm \, dv
\rk| , \EEQS
from which we deduce that
\DEQS \lqq{ \lk| I_1  (z)\rk|}
\\
 &\le & 2
(1+\theta)\,|U_2(z)|^ {-\theta-1 }\, |\xi|\, \lk| \int_{0}^\infty
\lk( |{U_1(v)}|^ {-\theta}+|U_2(z)|^ {-\theta} \rk) ^ {-\frac
1\theta-1}  \, dv \rk| . \EEQS
Again, substitution with $l=U_1(v)$
gives
\DEQS
\lqq{ \lk| I_1  (z)\rk|  \le  2 (1+\theta)\,|U_2(z)|^
{-\theta-1 }\, |\xi|\,}
\\ &&  \lk| \int_{0}^\infty \lk( |l|^
{-\theta}+|U_2(z)|^ {-\theta} \rk) ^ {-\frac 1\theta-1} {1\over
f_1\lk( U^{-1}_1(l)\rk)} \, dl \rk| . \EEQS Observe, we have
$$
U_1^{-1}(l)= \beta ^\frac 1\beta \, l^{-1/\beta}.
$$
Since $f_1(y)=y^{-1-\beta}$ we get for any $\gamma>0$
$$
g(l):= {1\over  f_1\lk( U^{-1}_1(l)\rk)} = C_\beta l ^{-\frac 1\beta -1}.
$$
Thus, we can write
\DEQS
\lk| I_1 (z)\rk|  &\le & 2
(1+\theta)C_\beta  |\xi|\, %|U_2(z)|^ {1+\theta}
  \int_{0}^\infty   \lk( \lk({ |l|\over |U_2(z)|}\rk) ^ {-\theta}+1 \rk) ^ {-\frac 1\theta-1}l ^{-\frac 1\beta -1}  \, dl .
\EEQS
Substitution gives
\DEQS
\lk| I_1  (z)\rk|  &\le & 2
(1+\theta)C_\beta  |\xi|\, % |U_2(z)|^ {1+\theta}
 \int_{0}^\infty   \lk( {u } ^ {-\theta}+1 \rk) ^ {-\frac 1\theta-1} (  u \, U_2(z) ) ^{-\frac 1\beta -1} \, U_2(z) \, du
\\&\le &
2
(1+\theta)C_\beta |\xi|\,  |U_2(z)|^ {-\frac 1\beta }
 \int_{0}^\infty   \lk( {u } ^ {-\theta}+1 \rk) ^ {-\frac 1\theta-1}   u  ^{-\frac 1\beta -1} \, du  .
\EEQS
Now consider,
\DEQS
\lk| I_2 (z)\rk|  &\le &  2
(1+\theta)\,|U_2(z)|^ {-\theta-1 }
\\
&&
 \lk| \int_{0}^\infty
i\xi y \lk( |U_1(y)|^ {-\theta}+|U_2(z)|^ {-\theta} \rk) ^
{-\frac 1\theta-2} \, |U_1(y)|^ {-\theta-1}
 f_1(y)  \, \, dy \rk|.
\EEQS
Substitution with $v=U_1(y)$ gives the estimate
\DEQS
\lk| I_2 (z)\rk|  &\le &  2
(1+\theta)\beta^{\frac{1}{\beta}}\,|U_2(z)|^ {-\theta-1 }\,|\xi|
\\
&&
 \lk| \int_{0}^\infty
  \lk( |v|^ {-\theta}+|U_2(z)|^ {-\theta} \rk) ^
{-\frac 1\theta-2} \, |v|^ {-\theta-\frac{1}{\beta}-1}
 \, dv \rk|
 \\&= &  2
(1+\theta)\beta^{\frac{1}{\beta}}\,|U_2(z)|^ {\theta}\,|\xi|
\\
&&
 \lk| \int_{0}^\infty
  \lk( \lk(\frac{|v|}{|U_2(z)|}\rk)^ {-\theta}+1 \rk) ^
{-\frac 1\theta-2} \, |v|^ {-\theta-\frac{1}{\beta}-1}
 \, dv \rk|.
\EEQS
Now take $u=\frac{|v|}{|U_2(z)|}$ as a substitution to obtain,
\DEQS
\lk| I_2 (z)\rk|  &\le &  2
(1+\theta)\beta^{\frac{1}{\beta}}\,|U_2(z)|^ {-\frac{1}{\beta}}\,|\xi|
\\
&&
  \int_{0}^\infty
  \lk( u^{-\theta}+1 \rk) ^
{-\frac 1\theta-2} \, u^ {-\theta-\frac{1}{\beta}-1}
 \, du.
\EEQS
Since $U_2(z)=\beta |z|^ {-\beta}$ we have
\DEQS
\lk| I_1(z)+I_2(z)\rk|  &\le & 2 \hat{C} |\xi|\,  |z|,
\EEQS
where $$\hat{C}=2(1+\theta)\lk(C_\beta\int_{0}^\infty   \lk( {u } ^ {-\theta}+1 \rk) ^ {-\frac 1\theta-1}   u  ^{-\frac 1\beta -1} \, du +\beta^{\frac{1}{\beta}}\int_{0}^\infty
  \lk( u^{-\theta}+1 \rk) ^
{-\frac 1\theta-2} \, u^ {-\theta-\frac{1}{\beta}-1}
 \, du\rk).$$
Since $\beta\le 1$,  we have
$$\int_{|z|\le 1} k_2(z)^p\nu_2(dz)=\int_{-1}^ 1  |z|^p \, |z|^{-\beta-1}\,
dz<\infty,
$$
for any $p>1$. This shows that the upper index of $\OPER _z$ is
$1$. Since $\alpha>1$, there exists a number $p>1$ such that
$$
{\beta^{+}\over \alpha^{-}}<\frac 1p.
$$
By the assumptions, the law of $X_0$ has a density function $F$ which is integrable and  $\rho_0(f)=\int_\RR \rho_0(x)\, f(x)\, dx$. Therefore we have $\rho_0\in L^2 (\RR)$.
Hence, by Theorem \ref{measurevalued} one can conclude the proof of Corollary 3.1.
\end{proof}
\appendix

\section{The Zakai Equation as a stochastic evolution equation}
\label{aA}

In this appendix we treat the Zakai equation as a stochastic
evolution equation on a Hilbert space and establish the
existence and uniqueness of its mild solution.
 For doing so, let $\mathfrak{X}$ be a Hilbert space, $A$ be a possibly unbounded operator generating an analytic
 $C_0$ semigroup $(T_A (t))_{t\ge 0}$ on $\mathfrak{X}$.
 %Usually, to define the stochastic integral, $\mathfrak{X}$ should be at least of $M$--type $p$.
% However, here we only formulate the problem and later $\mathfrak{X}$ will  be some Hilbert space.
Let $\eta$ be a time homogenous Poisson random measure with \levy measure $\nu$ on a measurable space
$(Z,\mathcal{Z})$  over a probability space
$(\Omega,\CG,(\CG_t)_{t\ge 0},\QQ)$ and $B=\{B(t):t\ge 0\}$ be a
$1$--dimensional Brownian motion defined over the same filtered
probability space. Let $f:\mathfrak{X}\to \mathfrak{X}$,
$\Sigma:\mathfrak{X}\to \mathfrak{X}$ be two mappings and
{$G:[0,T]\times \mathfrak{X}\times\RR\to\mathfrak{X}$} be a
progressively measurable mapping.
Consider the following equation with random initial data $u_0$:
\DEQSZ\label{eqn_mild}
\lk\{ \barray du (t) &=& \lk( Au (t) + f(u (t))\rk) \, dt + \Sigma(u(t))\, dB(t)%+ \Sigma(u_t)\, dB(t)
\\ &&{}+\int_\RR  G( t,u (t^-),z) \tilde{\eta}(dz,dt),\\
u (0)&=& u_0\in \mathfrak{X},\;\PP\;a.s.,
\earray \rk.
\EEQSZ
{where $\tilde{\eta}(dz,dt)=\eta(dz,dt)-\nu(dz)dt$ is the compensated Poisson random measure.} Now we define the concept of solution we have in mind.
\begin{defn}
We call a stochastic process $u=\{ u(t):t\ge 0\}$  a mild solution to
\eqref{eqn_mild}, if $u $ is \cadlag in $\mathfrak{X}$ and
satisfies $\PP$-a.s.
\DEQS \lqq{u(t) = u_0 + \int_0^t T_A({t-}r)
f(u(r)) \, dr}
\\ &&{} + \int_0^ t T_A(t-r) \Sigma(u(r))\, dB(r)  + \int_0^t \int_{\RR}T_A(t-r) G(r,u(r^-),z) \, \tilde \eta(dz,dr).
\EEQS
%Here, we assume that the integrals of the RHS are well defined.
\end{defn}
We state and prove the following result.
\newcommand{\pp}{q}

\begin{thm}\label{main}
Fix $\vr\in\RR$.
%, $L=(L_1,L_2)$ a $d$--valued \levy process, intensity measure $\nu$, with the marginals $\nu_1$ and $\nu_2$.
Let us assume that
\begin{itemize}
\item there exists some $\vr _0> -1$ such that $u_0\in H^ {\vr _0}_2(\RR)$, $\PP$ a.s.;
  \item the operator $A$ has  symbol $\psi$ with lower   Blumenthal--Getoor index $\alpha^-_0$;
%  \item $\nu_1$ and $\nu_2$ are supported by $(0,\infty)$;
\item there exists a $\delta_f<  {\alpha^ -_0}$ and a constant $C_f>0$ with %the mapping
$$
  \lk| f(x)-f(y)\rk|_{H^ {\vr-\delta_f}_2} \le C_f|x-y|_{H^\vr_2} ,\quad  x,y\in H^ \vr_2(\RR)
$$
\item there exists a  $\delta_\Sigma< \frac {\alpha^ -_0}2$ and  a
constant $C_\Sigma>0$ such that
$$
 \lk| \Sigma(x)-\Sigma(y)\rk|_{H^ {\vr-\delta_\Sigma}_2}\le C_\Sigma |x-y|_{H^\vr_2},\quad x,y\in H^ \vr_2(\RR);
$$
  \item there exists $\beta^+\ge 0$ and some $\pp\in[1,2]$ such that  the operator $G$ satisfies the following inequality
\DEQSZ\label{lip_G}
  \int_{|z|\le 1} \lk| G(r,x,z)-G(r,y,z)\rk|_{H^ {\vr-\beta^+}_2(\RR)}^ \pp\nu(dz) \le C_G|x-y|^ \pp_{H^\vr_2},
\EEQSZ
 $y,x\in  H^\vr_2(\RR)$ and for $|z|\ge 1$, $r\in[0,T]$,
\DEQSZ\label{lip_Gbig}
\\
\nonumber
 \lk| G(r,x,z)-G(r,x,z)\rk|_{H^ {\vr-\beta^+}_2(\RR)} \le C_G |x-y|_{H^\rho_2},\quad y,x\in  H^ \vr_2(\RR),\;r\in[0,T].
\EEQSZ
\end{itemize}
In addition, if %, $0\le \rho \le \frac 12$ and for some $\pp\in(1,2]$
 $$\vr -\vr _0<\frac 1 \pp,  \quad \mbox{ and } \quad {\beta^+\over \alpha^ -_0}<\frac 1\pp,
%
%\int_{|z_2|\le 1}  \lk|k(z_2)\rk|^ { \pp}\nu_2(dz_2)<\infty,
$$
then, there exists a mild solution $u$ belonging $\PP$-a.s.\ to  $\DD((0,T],H^\vr _2(\RR))\cap \DD([0,T];H_2^{\vr_0}(\RR))$ of
the stochastic evolution equation
\DEQSZ\label{eqn_mild111}
\\
\nonumber
\lk\{ \barray du (t) &=& \lk( A u (t)+ f(u (t))\rk) \, dt + \Sigma( u (t))\, dB(t)
\\ &&{}+\int_\RR G(t,u (t^-),z) \, \tilde{\eta}(dz,dt),\\
u (0)&=& u_0,\;\PP\;a.s. %\in \ \mathcal{M}(\RR).\text{  What is  $\mathcal{M}(\RR)$!! This should be $H^ {-\beta^ +}(\RR)$!!!}}
\earray \rk.
\EEQSZ
with random initial data $u_0\in H^ {\rho_0}_2(\RR)$.
\del{ with $u$.
and  we have for $\rho <\rho _0+\frac 1p$
$$
\EE^ \QQ\lk[ \int_0^ T |u(t)|_{H_2^ {\rho }}^ \rho\, dt\rk] <\infty.
$$
}
% \eqref{eqn_mild}

\end{thm}

\begin{proof}
First we tackle the case where the $q$--moments are bounded, i.e.\
we suppose
\DEQSZ\label{lip_g}
  \int \lk| G(r,x,z)-G(r,y,z)\rk|_{H^ {\vr-\beta^+}_2(\RR)}^ \pp\nu(dz) \le C_G|x-y|^ \pp_{H^\vr_2},
\EEQSZ
for $y,x\in  H^ \vr_2(\RR)$, $r\in[0,T]$. Let $\vr <\vr _0+\frac 1\pp  $ and
\DEQS
\lqq{\CM^ \pp _{\lambda,\vr } ([0,T]\times \Omega; \RR):= \biggl\{ u: [0,T]\times \Omega\to\RR,\,\biggr.
}&&\\ &&\lk.\mbox{ $u$ is progressively measurable and }\,  \EE \int_0^ T e^{-\lambda t} |u(t)|_{H^ \vr_2  (\RR)}^ \pp  \, dt <\infty\rk\}
\EEQS
equipped with the norm
$$
|u|_{\CM^ \pp  _{\lambda,\vr }} := \lk( \EE \int_0^ T e^{-\lambda t} |u(t)|_{H^ \vr _2 (\RR)}^ \pp  \, dt\rk) ^ \frac 1\pp ,\quad u\in \CM^ \pp _{\lambda,\vr } ([0,T]\times \Omega; \RR).
$$
Now, the existence of the mild solution will be established by making use of Banach fixed point Theorem (see e.g. \cite{ejp}).
For any $\lambda>0$ let us define the operator
$$\CI:\CM^ \pp_{\lambda,\vr } ([0,T]\times \Omega; \RR)\to \CM^ \pp_{\lambda,\vr } ([0,T]\times \Omega; \RR)
$$
%for $u\in \CM^ 2_{\lambda,\rho } ([0,T]\times \Omega; \RR)$ %for $t>0$
by
\DEQS
\lqq{
\CI(u)(t) =T_{A} (t) u _0+ \int_0^t T_A({t-}r) f(u(r)) \, dr+\int_0^t T_{A} (t-r)    \del{+ \int_0^t T_{A} (t-r) \, u(r) h \, dr }} &&
\\
&&{}\times\,\Sigma(u(r))  dB(r) + \int_0^t  \int_{|z|\le 1}  T_{A} (t-r) G(r, u(r^-),z) \, \tilde{\eta}(dz,dr)
\del{\\
&&{} + \int_0^t  \int_{|z|\ge 1} T_{A} (t-r) \CBB _{z_2} u(r) \, \eta(dz,dr)} ,\quad t\ge 0,
\EEQS
and $u\in \CM^ q_{\lambda,\vr } ([0,T]\times \Omega; \RR)$. First, we have to show that $\CI$ maps $\CM^ \pp_{\lambda,\vr }
([0,T]\times \Omega; \RR)$ into itself.\del{In this calculation, we use notation $C$ as a generic constant associate with the all estimates.} Since the symbol of
$\psi_0$ has a Blumenthal--Getoor lower index $\alpha_0^ -$,  Theorem 2.1
in \cite{hsym} implies that for $\gamma\ge 0$ and
$\delta\in\RR$
\DEQSZ\label{oben11}
\lk|T_{A}(t) u _0\rk|_{H^ {\delta-\gamma}_2(\RR) }\le {C t^{- \frac \gamma{\alpha^-_0}} } |u _0|_{ H_2^ {\delta }(\RR)}, \quad u _0\in { H_2^ {\delta}(\RR)}.
\EEQSZ
Hence,
\DEQS
\lqq{
\EE\int_0^T e^{-\lambda r} \lk|T_{A}(r) u _0\rk|^\pp_{H^ {\vr }_2 }\, dr}
\\
&\le&  C\EE\int_0^T e^{-\lambda r} r^{-\frac {\pp(\rho -\rho _0)}{\alpha^-_0}}  \lk|u _0\rk|^\pp_{H_2^ {\rho _0} }\, dr\le {C\over \lambda^{1-\frac {\pp(\rho -\rho _0)}{\alpha^-_0}}}
\,  \EE\lk|u _0\rk|^\vr_{H_2^ {\vr _0} }.
\EEQS
The Minkowski's integral inequality and the assumption regarding on $f$ give for the second term
\DEQS
\lqq{
\EE\int_0^ T e^{-\lambda t}\lk|\int_0^t T_A({t-}r) f(u(r)) \, dr\rk|_{H^ \vr_2}^\pp\, dt }
&&
\\&\le &\int_0^ T e^ {-\lambda t} \EE\lk(\int_0^t \lk|T_A({t-}r) f(u(r))\rk|_{H^ \vr_2} \, dr\rk)^\pp\, dt
\\&\le &C^q\int_0^ T \EE\lk(\int_0^t e^ {-\frac{\lambda (t-r)}{q}} (t-r)^ {-\frac{\delta_f}{\alpha_0^ -}} e^ {-\frac{\lambda r}{q}}\lk| f(u(r))\rk|_{H^ {\vr-\delta_f}_2} \, dr\rk)^\pp\, dt
\\&\le &(CC_f)^q\int_0^ T \EE\lk(\int_0^t e^ {-\frac{\lambda (t-r)}{q}} (t-r)^ {-\frac{\delta_f}{\alpha_0^ -}}e^ {-\frac{\lambda r}{q}}(1+\lk| u(r)\rk|_{H^ {\vr}_2} )\, dr\rk)^\pp\, dt
\EEQS
Applying Young's inequality for the convolution term gives
\DEQS
\lqq{
\EE\int_0^ Te^{-\lambda t} \lk|\int_0^t T_A({t-}r) f(u(r)) \, dr\rk|_{H^ \vr_2}^\pp\, dt }
&&
\\&\le &
(CC_f)^q\int_0^ T e^ {-\frac{\lambda r}{q}} r^ {-\frac{\delta_f}{\alpha_0^ -}}\, dr \cdot \int_0^ T e^ {-\lambda r}  \EE(1+\lk| u(r)\rk|_{H^ {\vr}_2} )^q \,\, dr
\\&\le &
{C_1\over \lambda^ {1-\frac{\delta_f}{\alpha_0^ -}}}
 \cdot \int_0^ T e^ {-\lambda r}  \EE(1+\lk| u(r)\rk|_{H^ {\vr}_2} )^q \,\, dr<\infty,
\EEQS
where $C_1=(CC_f)^qq^{1-\frac{\delta_f}{\alpha_0^ -}}\int_0^{\frac{\lambda T}{q}}e^{-\theta}\theta^{-\frac{\delta_f}{\alpha_0^ -}}\,d\theta$. For the third term, we get
\DEQS
 \EE\int_0^ T \lk[ \int_0^t   \lk|e^ {-\frac{\lambda(t-r)}{q}}  T_{A} (t-r) e^ {-\frac{\lambda r}{q}} \Sigma( u(r))  \, \rk|^ 2 _{H^ \vr_2 } dr \rk]^\frac \pp 2 \, dt
  \\
  \le   C^q\EE\int_0^ T \lk[ \int_0^t (t-r)^{-\frac{2\delta_\Sigma}{\alpha_0^-}} \,e^ {-\frac{2\lambda(t-r)}{q}}\,  e^ {-\frac{2\lambda r}{q}} \lk|    \Sigma (u(r))\, \rk|^ 2 _{H_2^ {\vr-\frac{\delta_\Sigma}{\alpha_0^-}} } dr \rk]^\frac \pp 2 \, dt.
\EEQS By the assumption on $\Sigma$ we can infer that
 \DEQS \ldots
\le (CC_\Sigma)^q\EE\,\int_0^ T \lk[ \int_0^t (t-r)^{-\frac{2\delta_\Sigma}{\alpha_0^-}}
\,e^ {-\frac{2\lambda(t-r)}{q}}\,  e^ {-\frac{2\lambda r}{q}}(1+ \lk|   u(r) \rk| _{H_2^
{\vr } })^ 2 dr \rk]^\frac \pp 2 \, dt. \EEQS Then applying Young's inequality
for the convolution \DEQS \ldots \le (CC_\Sigma)^q\int_0^ T
r^{-\frac{2\delta_\Sigma}{\alpha_0^-}} \,e^ {-\frac{2\lambda r}{q}}\,\, dr\cdot\int_0^T
e^ {-\lambda r}\EE(1+\lk|   u(r)\rk| _{H_2^ {\vr } })^q dr. \EEQS
Hence, we have \DEQS
 &&\EE\int_0^T e ^{-\lambda t}\lk| \int_0^t T_{A} (t-r) \Sigma(u(r))  dW(r) \rk|^q _{H_2^\vr} \, dt\\
&\le &{C_2\over \lambda ^{1-\frac{2\delta_\Sigma}{\alpha_0^-}} } \,  \int_0^T e ^{-\lambda t}\EE(1+\lk|u(t)\rk|)^q _{H_2^\vr} \,  dt<\infty,
\EEQS
where $C_2=(CC_\Sigma)^q(q/2)^{1-\frac{2\delta_\Sigma}{\alpha_0^ -}}\int_0^{\frac{2\lambda T}{q}}e^{-\theta}\theta^{-\frac{\delta_\Sigma}{\alpha_0^ -}}\,d\theta$.
It remains to calculate the fourth term.
By the assumptions on $G$ and $A$ we get
\DEQSZ\label{KEY}
\\
\nonumber
\int_\RR \lk| T_{A}(t)  G(r,x,z) \rk|^ \pp _{H^ \rho _2 (\RR^d )} \nu(dz) \le {C}^q\,  t^ {-\pp \, { \beta^+\over \alpha ^ -_0}}\,  (1+\lk|x\rk| _{H^ \vr _2})^ \pp ,\quad x\in H^ \vr_2 (\RR).
\EEQSZ
In particular,  if
$$2 \, {\beta^+ \over\alpha _0^ - }<1 %\quad \mbox{ and } \quad \frac \pp 2<\alpha^ -_0
$$
\del{then one can apply the Young inequality for convolution and get
\DEQS
\lqq{ \int_0^T e^{-\lambda t} \lk|\int_0^t  \int_\RR  T_{A} (t-r) G(r,u(r),z) \, \eta(dz,dr)\rk|_{H^\rho_2}^q\, dt }&&
\\
&\le &  \int_0^T e^{-\lambda t}\int_0^t  \int_\RR
 (t-r)^ {-\pp
\, { \beta^+\over \alpha ^ -_0}}\,  \lk| G(r,u(r),z
)\rk|_{H^\rho_2}^q\, \, dr \, dt
\\
&\le& {C\over \lambda ^{1-\pp \, {
\beta^+\over \alpha ^ -_0} }} \int_0^T e^{-\lambda t} \lk|
u(t)\rk|^q_{H^\rho_2}\, dt.
\EEQS}
 one can deal with the fourth term as follows. Consider firstly,
\DEQS
 \EE\int_0^ T \lk[ \int_0^t\int_\RR   \lk|e^ {-\frac{\lambda(t-r)}{q}}  T_{A} (t-r) e^ {-\frac{\lambda r}{q}} G(r,u(r),z)  \, \rk|^ 2 _{H^ \vr_2 }\nu(\,dz)\,dr \rk]^\frac \pp 2 \, dt
  \\
  \le (CC_G)^q\EE\int_0^ T \lk[ \int_0^t (t-r)^{-\frac{2\beta^+}{\alpha_0^-}} \,e^ {-\frac{2\lambda(t-r)}{q}}\,  e^ {-\frac{2\lambda r}{q}}(1+\lk|u(r)\rk| _{H^ \vr _2})^2  dr \rk]^\frac \pp 2 \, dt.
  \EEQS

By following similar argument as in Brownian term, we have \DEQS
 &&\EE\int_0^T e ^{-\lambda t}\lk| \int_0^t\int_\RR T_{A} (t-r)G(r,u(r^-),z) \, \tilde{\eta}(dz,dr) \rk|^q _{H_2^\vr} \, dt\\
&\le &{C_3\over \lambda ^{1-\frac{2\beta^+}{\alpha_0^-}} } \,  \int_0^T e ^{-\lambda t}\EE(1+\lk|u(t)\rk|)^q _{H_2^\vr} \,  dt<\infty,
\EEQS
where $C_3=(CC_G)^q(q/2)^{1-\frac{2\beta^+}{\alpha_0^ -}}\int_0^{\frac{2\lambda T}{q}}e^{-\theta}\theta^{-\frac{\beta^+}{\alpha_0^ -}}\,d\theta$.
Then collecting all estimates yields
that $\CI$ maps $\CM^ \pp_{\lambda,\vr } ([0,T]\times \Omega;
\RR)$ into itself.

Next, we will show that there exists a $\lambda>0$ such that the
operator $\CI:\CM^ \pp_{\lambda,\vr } ([0,T]\times \Omega;
\RR)\to \CM^ \pp_{\lambda,\vr } ([0,T]\times \Omega; \RR)$
is a strict contraction. To show the claim, let $u,v\in \CM^ \pp _{\lambda,\vr } ([0,T]\times \Omega; \RR)$. Then
\DEQS
\lqq{
\lk| \CI(u)-\CI(v)\rk|^ \pp _{\CM^ \pp _{\lambda,\vr }}} &&
\\
&\le &   \int_0^ T e^ {-t\lambda} \EE \lk|\int_0^t T_{A} (t-r)  \lk(\Sigma(u(r))-\Sigma(v(r))\rk)  dB(r) \rk|^ \pp _{H^ \vr_2 } \, dt
\\
&&{} + \int_0^ T e^ {-t\lambda} \EE \lk| \int_0^t  \int_\RR  T_{A} (t-r) \lk(  G(r,u(r^-),z)- G(r, v(r^-),z)\rk)  \eta(dz,dr)\rk|_{H^ \vr_2 } ^ \pp  \, dt
\\
&& \int_0^ T e^ {-t\lambda} \EE \lk|\int_0^t T_{A} (t-r)  \lk(f(u(r))-f(v(r))\rk) \, dr \rk|^ \pp _{H^ \vr_2 } \, dt
.
\EEQS

\del{Observe, over $(\Omega,\CF,\QQ)$ the process $B$ is a Brownian motion. Hence,
the Burkholder and Minkowski inequality gives
\DEQS
\lqq{
\lk| \CI(u)-\CI(v)\rk|^ \pp _{\CM^ \pp _{\lambda,\rho }}} &&
\\
&\le & C     \int_0^ T e^ {-t\lambda} \lk[ \int_0^t  \EE \lk|T_{A} (t-r) \lk[\Sigma(u(r))-\Sigma(v(r))\rk]\, \rk|^ 2 _{H^ \rho_2 } dr \rk]^\frac \pp 2 \, dt
\\
&& + C\int_0^ T e^ {-t\lambda}  \int_0^t \int_\RR \EE \lk|  T_{A} (t-r)  \lk(  G(r,u(r),z)- G(r, v(r),z)\rk) \rk| ^ \pp  _{H^ \rho _2}\nu(dz)\, dr\, dt
\\
&&C\lk(  \int_0^ T e^ {-t\lambda} \EE \lk(\int_0^t \lk|T_{A} (t-r)  \lk(f(u(r))-f(v(r))\rk)\rk| \, dr \rk)^ \pp _{H^ \rho_2 } \, dt \rk.
.
\EEQS
For the second term, we get
\DEQS
  \int_0^ T \lk[ \int_0^t  \EE \lk|e^ {-(t-r)\lambda}  T_{A} (t-r) e^ {-r\lambda} \lk[ u(r)-v(r)\rk]\, g \, \rk|^ 2 _{H_2^ \rho } dr \rk]^\frac \pp 2 \, dt
  \\
  \le   \int_0^ T \lk[ \int_0^t  \EE (t-r)^{-2\delta_\Sigma/\alpha_0^-} \,e^ {-(t-r)\lambda}\,  e^ {-r\lambda} \lk|    \lk[ u(r)-v(r)\rk]\, g \, \rk|^ 2 _{H_2^ {\rho -\delta_\Sigma} } dr \rk]^\frac \pp 2 \, dt.
\EEQS
Therefore
\DEQS
\ldots \le
C\,   \int_0^ T \lk[ \int_0^t  \EE (t-r)^{-2\delta_\Sigma/\alpha_0^-} \,e^ {-(t-r)\lambda}\,  e^ {-r\lambda} \lk|   u(r)-v(r)\, \rk|^ 2 _{H^ {\rho }_2 } dr \rk]^\frac \pp 2 \, dt.
\EEQS
The Young inequality for convolution
gives
\DEQS
\ldots \le
C\,   \lk( \int_0^ T r^{-2\delta_\Sigma/\alpha_0^-} \,e^ {-r\lambda}\,\, dr\rk) \lk(\int_0^T    e^ {-r\lambda} \lk|   u(r)-v(r)\, \rk|^ 2 _{H^ {\rho } _2} dr\rk).
\EEQS
Assumption \ref{lip_g} %and \ref{lip_f}
implies
\DEQS
\lqq{
\lk| \CI(u)-\CI(v)\rk|^ \pp _{\CM^ \pp _{\lambda,\rho }}} &&
\\
&\le & C     \int_0^ T e^ {-t\lambda} \lk[ \int_0^t  \EE \lk|u(r)-v(r)\rk|^ 2 _{H^ \rho_2 } dr \rk] ^\frac \pp 2\, dt
\\
&& + C\int_0^ T  \int_0^t (t-r) ^ {-\pp \frac {\beta^+}{\alpha_0^-}} e^ {-(t-r)\lambda} e^ {-r\lambda}  \EE \lk|   u(r)-v(r) \rk| ^ \pp  _{H_2^ \rho } \,dr\, dt
\\
&& + C\int_0^ T \lk( \int_0^t (t-r) ^ {- \frac {\beta^+}{\alpha_0^-}} e^ {-(t-r)\lambda} e^ {-r\lambda}  \EE \lk|   u(r)-v(r) \rk| ^ \pp  _{H_2^ \rho } \,dr\rk)^ \pp \, dt
\\
&\le & C     \int_0^T    e^ {-r\lambda} \lk|   u(r)-v(r)\, \rk|^ 2 _{H_2^ {\rho } } dr
\\
&& + C  \int_0^ T  \int_0^t(t-r) ^ {-\pp \frac {\beta^+}{\alpha_0^-}} e^ {-(t-r)\lambda} e^ {-r\lambda}  \EE \lk|   u(r)-v(r) \rk| ^ \pp  _{H_2^ \rho } \,dr\, dt
\\
&& + C  \int_0^ T  \lk( \int_0^t(t-r) ^ {- \frac {\beta^+}{\alpha_0^-}} e^ {-(t-r)\lambda} e^ {-r\lambda}  \EE \lk|   u(r)-v(r) \rk| ^ \pp  _{H_2^ \rho } \,dr\rk) ^\pp \, dt
 .
\EEQS
The Young inequality for convolution gives}%for $\gamma=1- p\frac {\beta^+}{\alpha_0^-}$
Then by following similar arguments as in previous calculation, we can easily show that,
$$\lk| \CI(u)-\CI(v)\rk|^ \pp _{\CM^ \pp _{\lambda,\vr }}\le {\hat{C}\over \lambda ^ \varepsilon} \lk| u-v\rk|^ \pp _{\CM^ \pp_{\lambda,\vr }} ,
$$
where $\hat{C}=\max\{C_1,C_2,C_3\}$ and $\varepsilon=\min\{1-\frac{\delta_f}{\alpha_0^ -},1-\frac{2\delta_\Sigma}{\alpha_0^ -},1-\frac{2\beta^+}{\alpha_0^-}\}$.
Hence $\CI$ is a strict contraction for $\lambda $ sufficiently
large.

To conclude the proof of the theorem we show that
$u\in\DD((0,T],H_2^\vr (\RR))\cap \DD([0,T];H_2^{\vr _0}(\RR))$.
For this purpose, we consider the stochastic convolution term with
respect to the Brownian term, i.e.\
$$
\int_0^t T_{A} (t-r) \Sigma( u(r)) \, dB(r).
$$
The continuity of this term follows by \cite[Theorem 5.9, p.\ 127]{DZ}.
It remains to investigate the \cadlag property of
$$
 \int_0^t  \int_\RR  T_{A} (t-r) G(r,u(r^-),z) \, \tilde{\eta}(dz,dr).
 $$
But
Proposition 1.3 in \cite{jan} leads to
\DEQS
&&
\EE\lk| \int_0^t  \int_\RR  T_{A} (t-r) G(r,u(r^-),z)\, \tilde{\eta}(dz,dr)\rk|^\pp_{H_2^{\vr -\beta^+}} \\&\le&
\EE \int_0^t  \int_\RR \lk| G(r,u(r),z)\rk|^\pp_{H_2^{\vr -\beta^+}} \nu(dz)\, dr.
\EEQS
Since for any $z\in\RR_+^0$, $G(.,.,z):H_2^ {\vr }(\RR)\to H_2^ {\vr -\beta^+}(\RR)$ is bounded,
the \cadlag property follows.

\del{In order to tackle the large jumps, we would like to recall that
the waiting times between the large jumps are exponentially
distributed. Thus, if $\{T_i:i\in\NN\}$ are the jump times with
the size of the jumps are larger than one, we need to apply the
above result on each time interval and glue together the
solutions. That is, for $i\in\NN$, let $u_i$ be a solution to
\DEQS %Z\label{eqn_mild111}
\lk\{ \barray du_i (t) &=&  A u_i (t) \, dt + \Sigma( u_i (t))\, dB(t)  +\int_\RR G(t,u _i(t),z) \, \tilde{\eta}(dz,dt),\\
u (0)&=& u_{i-1}(T_{i-1})+ \int_{|z|>1} G(T_{i-1},u_{i-1}(T_{i-1}),z)\tilde{\eta}(dz,\{T_{i-1} \}). %\in \ \mathcal{M}(\RR).\text{  What is  $\mathcal{M}(\RR)$!! This should be $H^ {-\beta^ +}(\RR)$!!!}}
\earray \rk.
\EEQS
If $\{ G(r,x,z); |z|>1,x\in H^ \vr_2(\RR)\}\subset H^ {\vr_0}_2(\RR)$, for each $i\in\NN$ the solution exists and is unique.
Now, put
$$
u(t) = u_i(t),\quad T_{i-1}\le t< T_i,\quad i\in\NN.
$$
Now, $u$ is the unique solution to the problem  \eqref{eqn_mild111} and $u\in
\DD([0,T];H_2^{\vr_0}(\RR))$. Note that $u\in\DD((0,T],H^\vr
_2(\RR)) $ if $\{ G(r,x,z); |z|>1,x\in H^
\vr_2(\RR)\}\subset H^ {\vr}_2(\RR)$.}
In previous analysis, we assumed that $q$--moments are bounded of the jump term (see \eqref{lip_g}) to construct the solution to \eqref{eqn_mild111} using fixed point method. In general, we should only consider small jumps with the assumption \eqref{lip_g} and prove the existence of the solution by using fixed point method, since if we allow large jumps to occur, then the corresponding jump integral may blow up and the fixed point method will collapse.
Notice that the random jump times with jump size larger than one are independent of the $\sigma$-algebra generated by small jumps (size less than one) and Brownian motion.
In particular, the Poisson random measure is independently scattered, or in other words,  for any $U\in\mathcal{B}(\RR)$ the processes
$\eta( U\cap (-1,1)\times [0,t])$ and $\eta( U\cap \RR\setminus (-1,1)\times [0,t])$ are independent.
Therefore, now we assume that \eqref{lip_g} holds with only small jumps (size less than one). Let $\{T_i:i=1,\ldots,n\}$ be the random jump times (stopping times) with the size of the jumps are larger than one.
Previous analysis guarantees that there exists a $\hat{u}\in\DD((0,T_1),H_2^\vr (\RR))\cap \DD([0,T_1);H_2^{\vr _0}(\RR))$, which solves
 \DEQSZ\label{eqn_mi1}
\lk\{ \barray d\hat{u} (t) &=& \lk( A\hat{u} (t) + f(\hat{u} (t))\rk) \, dt + \Sigma(\hat{u}(t))\, dB(t)%+ \Sigma(u_t)\, dB(t)
\\ &&{}+\int_{|z|<1}  G( t,\hat{u} (t^-),z) \tilde{\eta}(dz,dt)-\int_{|z|\geq1}  G( t,\hat{u} (t),z) \nu(dz)dt,\\
\hat{u} (0)&=& \hat{u}_0\in \mathfrak{X},\;\PP\;a.s.
\earray \rk.
\EEQSZ
 We follow interlacing criteria (see Theorem 2.5.1 in \cite{applebaum}) to construct the solution over whole interval $[0,T]$.

  Now we recursively construct the solution $u=u_n$ of \eqref{eqn_mild111} over whole interval $[0,T]$ as follows.
Define on $[0,T_1]$
\begin{equation}
u_1(t) =
\begin{cases}
\hat{u}(t) & \text{for $t<T_1$}\\
\hat{u}(T_1^-)+G(T_1^-,\hat{u}(T_1^-),\Delta P(T_1))  & \text{for $t=T_1$},
\end{cases}
\label{interlacing1}
\end{equation}
where $P(t)=\int_{|z|\geq 1}z\eta(dz,dt)$ is the compound Poisson process.
Now suppose that \\$\PP\left\{\omega\in\Omega: T_1<\infty\right\}=1$. Define $\bar{u}(0)=u_1(T_1)$, $\bar{B}(t)=B(T_1+t)$, $\bar{\eta}(.,t)=\eta(.,T_1+t)$ and $\bar{\mathscr{F}}_t=\mathscr{F}_{T_1+t}$. Let $\bar{P}(t)=\int_{|z|\geq 1}z\bar{\eta}(dz,dt)$ be the compound Poisson process which starts from time $T_1$.

Since we don't have jumps with size larger than one during the time interval $(T_1,T_2)$, from previous analysis there exists a solution  $\bar{u}(t-T_1)\in\DD((T_1,T_2),H_2^\vr (\RR))\cap \DD([T_1,T_2);H_2^{\vr _0}(\RR))$. Then,
\begin{equation}
u_2(t) =
\begin{cases}
u_1(t) & \text{for $t\leq T_1$}\\
\bar{u}(t-T_1) & \text{for $T_1 \leq t\leq T_2$}\\
\bar{u}((T_2-T_1)^-)+G((T_2-T_1)^-,\bar{u}((T_2-T_1)^-),\Delta \bar{P}(T_2))  & \text{for $t=T_2$}
\end{cases}
\label{interlacing2}
\end{equation}
\noindent Since we have a finite number of large jumps with size bigger than one over $[0.T]$ almost surely, by repeating the above process $n$ times, we can obtain $u=u_n\in\DD((0,T],H_2^\vr (\RR))\cap \DD([0,T];H_2^{\vr _0}(\RR))$ which solves \eqref{eqn_mild111}.
\end{proof}

\section{\levy  Copulas}\label{levy_copula}
\label{aB}

L\'evy copulas is a general concept to capture jump dependency in
multivariate L\'evy processes and is widely used in finance. In
this section, we only recall short facts about copulas, pair
copulas, L\'evy processes, and the L\'evy copula concept. Detailed
treatment of copulas and L\'evy copulas can be found in
\cite{cher,nelsen,mal} and \cite{tankov,tankov1,tankov2}.

Let $L_1$ and $L_2$ be two %spectrally {WHY RESTRICTING TO SPECTRALLY?} positive
L\'evy processes with \levy measures $\nu_1$ and $\nu_2$. Before
introducing the L\'evy copulas, let us introduce the extended tail
integrals $U_1$ and $U_2$.

First, we need following function associated with any $z\in\RR\setminus\{0\}$:
\DEQS
\mathcal{I} (z) =\bcase (z,\infty)\,& z>0,
\\
(-\infty,z), & z<0.\ecase
\EEQS
In the same way as the distribution of a random vector can be represented by its distribution
function, the \levy  measure of a \levy  process will be represented by its tail integral.

Now, the tail integral of a $2$--dimensional process can be defined for $i=1,2$ by
\DEQSZ\label{tail11}
U_i(z)=\bcase  \sgn(z) \, \nu_i\lk( \mathcal{I}(z)\rk),&\mbox{ if }\quad z\in\RR\setminus\{0\},
\\ 0&\mbox{ if } z=\infty,
\ecase
\EEQSZ
and their generalized inverse, given by
$$ U^ {\leftarrow}_i (z) := \sup\{ x\ge 0\mid U_i(x)=z\}, \quad z\ge 0,\quad i=1,2.
$$

Dependence of jumps of a multivariate \levy process can be described
by a \levy copula which couples the marginal tail integrals. In particular,
let $L$ be a two dimensional  \levy process, $\nu$ is its intensity measure and $U$ is the tail integral defined by
\DEQSZ\label{tail111}
U(z)=\prod_{i=1}^2  \sgn(z_i) \, \nu\lk(\prod_{i=1}^2 \mathcal{I}(z_i)\rk),\quad z=(z_1,z_2)\in (\RR\setminus\{0\}\cup\{\infty\})^2.
\EEQSZ
Now, $L$ can be seen as  two  L\'evy processes linked together by the  mapping
$
H:\RR^2\to\RR,
$
defined as
$$
U(z_1, z_2) = H(U_1(z_1), U_2(z_2)),\quad z_1,z_2\in\RR\setminus\{0\} \cup\{\infty\}.
$$
For example, if $L_1$ and $L_2$ are independent positive \levy processes, the copula $H$ is given by (see \cite[Theorem 4.6]{tankov2})
$$
H_\perp(z_1,z_2) = z_1 1_{z_2=\infty}+ z_2 1_{z_1=\infty}, \quad z_1,z_2\in\RR^ +\cup\{\infty\}.
$$
If $L_1$ and $L_2$ are completely dependent, the copula $H$ is given by
$$
H_\parallel(z_1,z_2) = \min(|z_1|,|z_2|)1_K(z_1,z_2) \sgn(z_1)\sgn(z_2), \quad z_1,z_2\in\RR.
$$
where $K=\{ (z_1,z_2)\in \RR^ 2 : \sgn(z_1)=\sgn(z_2)\}$.

A Sklar  type Theorem (see \cite{tankov}) ensures the existence
and uniqueness of a \levy copula given a \levy process, and vice
versa. To be more precise, it says that for each $2$--dimensional
\levy process with intensity $\nu$ and  marginal \levy measures $\nu_i$, $i=1,2$,  %(z_i,\infty)=\nu( (0,\infty)\times \cdots \times (z_i,\infty)\times \cdots  \times (0,\infty))$
one can associate a \levy copula $H$ such that
\DEQSZ\label{eqn11}
U(z_1, z_2) = \sgn(z_1)\sgn(z_2)\, H(U_1(z_1), U_2(z_2)),\quad z_1,z_2\in \RR\setminus \{0\}\cup\{\infty\}. %\ge 0.
\EEQSZ
Here $U$ and $U_i$, $i=1,2$, denotes the tail integrals defined by \eqref{def_tail_n} and \eqref{tail1} respectively.

Conversely, if $H$ is a \levy copula and $U_1,U_2$ are marginal tail integrals
of two \levy processes, Equation \ref{eqn11} defines the tail integral
of a $2$-dimensional  \levy process, where $U_1, U_2$ are the
tail integrals of its components.

As an example, let us consider Clayton \levy  copula.

\begin{ex}\label{clayton}
For a %spectrally positive,
$2$-dimensional \levy  processes
the Clayton copula is given on $\RR^2 $ by (see e.g.\ \cite{tankov1,tankov2})
\DEQSZ\label{clayton}
 H(z_1,z_2) =
\lk( \frac 12|z_1|^{-\theta}+\frac 12 |z_2|^{-\theta}\rk) ^{-\frac 1\theta}\lk( \beta 1_{z_1\cdot z_2>0} + (1-\beta) 1_{z_1\cdot z_2<0}\rk) ,\quad z_1,z_2\in\RR.
\EEQSZ
The parameter $\theta > 0$ determines the dependence of the jump sizes, where larger values
of $\theta$ indicate a stronger dependence, smaller values of $\theta$ indicate independence.
The parameter $\beta$ determines the dependence of the
sign of jumps: when $\beta=1$, the two components always jump in the same direction, and
when $\beta= 0$, positive jumps in one component are accompanied by negative jumps in the
other and vice versa. For intermediate values of $\beta$, positive jumps in one component can
correspond to both positive and negative jumps in the other component. The parameter
$\theta$ is responsible for the dependence of absolute values of jumps in different components.
\end{ex}

To give the connection between copulas and \levy copulas  let us
define the survival copula. Let $F:\RR^2\to [0,1]$ be a distribution
function and $\bar F (x,y) = 1-F(x,y)$. Let $F_1$ and $F_2$ be the
marginal distributions, $\bar F_1=1-F_1$ and $\bar F_2=1-F_2$ be
the marginal tail functions respectively. Now, one can define the survival copula associated to $F$ by
$$
\bar C(u,v) := \bar F(\bar F^ {-1}_1(u), \bar F^ {-1}_1(v)), \quad (u,v)\in [0,1]^ 2.
$$
Since $C( u,1)=u$ and $C(1,v)=v$, we get $\bar C(0,u)=u$ and $\bar C(v,0)=v$.

\subsection{Finite \levy  measure and copula}\label{Append 1}

For simplicity, let $L=(L_1,L_2)$ be a two dimensional  \levy process with only positive jumps and with marginal \levy measures  $\nu_1$, $\nu_2$ and copula $H$. Here, we assume that
$\nu_1$ and $\nu_2$ are two  \levy  measures with $\nu_1((0,\infty))=\lambda_1$, $\nu_2((0,\infty))=\lambda_2$. We also assume that $H$ is twice differentiable and $\nu_1$, $\nu_2$ have densities with respect to Lebesgue measure on $\RR\setminus\{0\}$.
We will consider only copula, such that $L_1$ and $L_2$ have only common jumps.

Let $(\CF^1_t)_{t\ge 0}$ be the filtration generated by $L_1$ and $(\CF^ 2_t)_{t\ge 0}$ the filtration generated by $L_2$.
We are interested in the jumps  of $L_1$ given the jumps of $L_2$.
Since
$$
\nu((z_1,\infty),(z_2,\infty))= H( U_1(z_1),U_2(z_2))
$$
it follows that
\DEQS
\nu(dz_1,dz_2) &=& {\partial^ 2 \over \partial u_1 \partial u_2} H(u_1,u_2)\Bigg|_{u_1=U_1 (z_1)\atop u_2=U_2 (z_2)}
%{\partial U_1 (z_1)\over \partial z_1} {\partial U_2 (z_2)\over \partial z_2}\,dz_1\,dz_2 = {\partial^ 2 \over \partial x_1 \partial x_2} H(u_1,u_2)\Bigg|_{u_1=U_1 (z_1)\atop u_2=U_2 (z_2)}
\nu_1(dz_1)\, \nu_2(dz_2).
%\\ &&{}+{\partial \over \partial u_1 } H(u_1,\infty)\Bigg|_{u_1=U_1 (z_1)}\nu_1(dz_1)+
%{\partial \over \partial u_2 } H(\infty,u_)\Bigg|_{u_2=U_1 (z_2)}\nu_2(dz_2).
\EEQS
Substitution gives
\DEQS
\nu(\RR\times \RR)&=& \int_0^\infty \int_0^\infty \nu(dz_1,dz_2) = \int_0^{\lambda_1}  \int_0^{\lambda_2}  {\partial ^2 H(u_1,u_2)\over \partial u_1\,\partial u_2} \, du_1\, du_2
\\
&=& H(0,0) - H(\lambda_1,0)-H(0,\lambda_2) + H(\lambda_1,\lambda_2) %:=\lambda_H.
\\
&=&  H(\lambda_1,\lambda_2) :=\lambda_H.
\EEQS
 Since $\nu_1$ and $\nu_2$ are finite,
it follows that $L(t)$ can be represented by
the following sum
$$
L(t) = \sum_{n=1} ^{N(t)} Y_n, %+\sum_{n=1} ^{N_1(t)}( Y_{n,1}^ {\ind},0)+\sum_{n=1} ^{N_2(t)} (0,Y_{n,2}^ {\ind}), %\Delta L_n,
$$
where $N=\{N(t):t\ge 0\}$ is a Poisson process with intensity
$\lambda_H$ and $\{ Y_n=(Y_{n,1},Y_{n,2})
:n\in\NN\}$ is a family of $\RR^2 $--valued independent random
variables with distribution function $ {\nu/\lambda_H} $.
\del{In addition, for $i=1,2$,
$N_i=\{N_i(t):t\ge 0\}$ are  Poisson processes
with intensity $\lambda_i$ and $\{Y_{n,i}^{\ind}:n\in\NN\}$ are families of independent
$\nu_i/\lambda_i$ distributed random variables.}
Calculating the Fourier transform one can easily see
\DEQS
\EE e^{i x L(t)} &=& \sum_{k=1}^\infty \EE\lk[ e ^{\sum_{n=1}^k i x Y_n}\mid N(t)=k\rk]\PP\lk( N(t)=k\rk)
\\
&=& \exp(-\lambda_H t) \sum_{k=1}^\infty \frac {(\lambda_H t)^k}{k!} \EE\lk[ e ^{i x Y_1}\rk] ^{k} =   \exp(-\lambda_H t)\,  \exp\lk( t\int_{\RR^2 } e^{ixy}{\nu(dx\times dy)}\rk)
\\
&=& \exp\lk( t\int_{\RR^2 } \lk( e^{ixy}-1\rk) {\nu(dx\times dy)}\rk)
.
\EEQS

We are
interested in the conditional distribution of the jumps in the first variable, given the jumps in the second variable,
 i.e.\ $Y_{n,1}$, given the projection onto
the second axis, i.e.\ $Y_{n,2}$.

If $\bar C$ is the survival copula of $Y_n$, i.e.\
$$
\bar C(u_1,u_2) =\bar F( \bar F_1^ {-1}(u_1),\bar F^ {-1}_2(u_2)),\quad u_1,u_2\in[0,1],
$$
with $\bar F_i(x_i)= U_i(x_i)/\lambda_i$, then
$$
\bar C(u_1,u_2) =\bar F( U_1^ {-1}(\lambda_1u_1),U_2^ {-1}(\lambda_2 u_2))=\frac 1{\lambda_H} U( U_1^ {-1}(\lambda_1u_1),U_2^ {-1}(\lambda_2 u_2))
$$
and, by the definition of the \levy copula $H$,
$$
\bar C(u_1,u_2) =\frac 1{\lambda_H} H(\lambda_1u_1,\lambda_2 u_2), \quad u_1,u_2\in[0,1].
$$

Fix $\ep>0$ and let us assume that we have a L\'evy measure with infinite activity and that we cut of all jumps whose projection onto one of the two axis
is smaller than $\ep$. Then we have
$$
\nu( (\ep,\infty)\times (\ep,\infty))= H(U^{-1}_1(\ep),U^{-1}_2(\ep))\big|_{u_1=U_1^{-1}(\ep)\atop u_2=U_2 ^{-1}(\ep)} = \bar C( \ep,\ep). %u_1,u_2)\big|_{u_1=U_1^{-1}(\ep)\atop u_2=U_2 ^{-1}(\ep)}.
$$
This gives us the scaling property
$$
\frac 1{\lambda} H(\lambda u_1,\lambda u_2)= H(u_1,u_2), \quad u_1,u_2\in\RR \setminus\{0\},
$$
for $\lambda= H(\ep,\ep)$.
\begin{prop}
Let us assume that $\lambda_1=\lambda_2=\lambda=H(\ep,\ep)$ and let us assume that the Copula $H$ satisfies the following
scaling property:
$$
\frac 1{\lambda} H(\lambda u_1,\lambda u_2)= H(u_1,u_2), \quad u_1,u_2\in\RR \setminus\{0\}.
$$
Let us define %
$$
h(u_1,u_2) := { {\partial ^ 2 \over \partial u_1\partial u_2} H(u_1,u_2) % \Big|_{u_1=U_1 (z_1)\atop u_2=U_2 (z_2)}
}\, . %\,{\nu_1(dz_1)}.
$$
Then, the conditional probability of $\Delta_t L_1$ given $\Delta_t L_2$ is represented by
$$
\PP\lk( \Delta_t L_1=z_1\mid \Delta_t L_2=z_2\rk) = \,
 h(u_1,u_2)\Big|_{u_1=U_1 (z_1)\atop u_2=U_2 (z_2)} \,{ \nu(dz_1)}.
$$
\end{prop}
\begin{proof}
 The formula can be shown by direct calculations. In particular,
\DEQS
\PP\lk( \Delta_t L_1=z_1\mid \Delta_t L_2=z_2\rk)  &=& { \PP\lk( \{\Delta_t L_1=z_1\}\cap \{ \Delta_t L_2=z_2\}\rk)\over \PP(\{\Delta_t L_2=z_2\})}
\\
 = { {\partial ^ 2 \over \partial z_1\partial z_2} \bar F \lk(z_1,z_2\rk) \over {\partial  \over \partial z_2} \bar F (0,z_2) }
&=&{ {\partial ^ 2 \over \partial z_1\partial z_2} \bar F \lk(z_1,z_2\rk) \over {\partial  \over \partial z_2} \bar F_2 (z_2) }
\\
 = { {\partial ^ 2 \over \partial z_1\partial z_2} \bar C \lk(F_1(z_1),F_2(z_2)\rk) \over {\partial  \over \partial z_2} \bar F (0,z_2) }
&=&\frac 1 {\lambda}{ {\partial ^ 2 \over \partial z_1\partial z_2} H \lk(\lambda_1F_1(z_1),\lambda_2F_2(z_2)\rk)  \over {\partial  \over \partial z_2} \bar F_2 (z_2) }
.
\EEQS
Substituting $\bar F_i(x_i)= U_i(x_i)/\lambda_i$  we get
\DEQS
\ldots &=& { \lambda_2\over \lambda}\, { h(u_1,u_2)  \Big|_{u_1=U_1 (z_1)\atop u_2=U_2 (z_2)}{\partial  \over \partial z_1} U_1 (z_1) {\partial  \over \partial z_2} U_2 (z_2)
\over {\partial  \over \partial z_2} U_2 (u_2) }
\\ &=& { \lambda_2\over \lambda} \, { h(u_1,u_2)  \Big|_{u_1=U_1 (z_1)\atop u_2=U_2 (z_2)}} \, \nu_1(dz_1)
 .\hspace{3cm}
 \EEQS
\end{proof}

\subsection{Copula and $\sigma$--finite L\' evy measures}\label{Append 2}

Let us assume that the $\nu_1$ and $\nu_2$ are two \levy  measures
with infinite measure.

Let $\nu$ be a $\sigma$-finite \levy measure and $L$ the corresponding \levy process. Here we consider $L$ with only positive jumps.
Cutting off the jumps smaller than $\ep$, the corresponding \levy process $L^ \ep$
can be written as follows.
$$
L^ \ep := \sum_{i=1} ^{N_\ep (t)} Y_{i,\ep},
$$
where $N_\ep$ is a Poisson point process with parameter $\nu(\RR^2_+ \setminus (0,\ep)\times (0,\ep))$ and
$\{ Y_{i,\ep}:i\in\NN\}$ are independent identical distributed random variables with survival function
\DEQSZ\label{barF}
\bar F_{\ep} (x,y) = { U \lk(x,y\rk)\over U(\ep,\ep)},\quad x,y\ge \ep.
\EEQSZ

Now, the aim is to express the survival copula of the two dimensional random variable $Y_{i,\ep}$
by the \levy copula $H$ and vice versa.
The survival copula $\bar C_\ep$ of $Y_{i,\ep}$   is given by
$$
\bar C_\ep (u,v) = \bar F_\ep ( \bar F_{1,\ep}^ {-1} (u), \bar F_{2,\ep}^ {-1} (u)),\quad u,v\in [0,1].
$$
Since
$$
 \bar F_{i,\ep}(x) = { U_{i,\ep}(x)\over U_i(\ep)},\quad i=1,2, %\mbox{ and }  \bar F_{2,\ep}(x) = { U_{2,\ep}(x)\over U_2(\ep)},
 $$
 where $U_{i,\ep}(x)=\nu_i([x,\infty))$ for $x\geq\ep$.
It follows that
$$
 \bar F_{i,\ep}^ {-1}(u)= U_{i,\ep}^{-1}(   U_i(\ep) u),\quad u\in [0,1],\, i=1,2.
 $$
Therefore
$$
\bar C_\ep (u,v) = \bar F_\ep \lk(  U_{1,\ep}^{-1}(   U_1(\ep) u),U_{2,\ep}^{-1}(   U_2(\ep) v)\rk).
$$
Next, \eqref{barF} implies that
$$
\bar C_\ep (u,v) =  { U\lk(  U_{1,\ep}^{-1}(   U_1(\ep)
u),U_{2,\ep}^{-1}(   U_2(\ep) v)\rk)\over U\lk(  \ep,\ep \rk)}.
$$
Finally, by the definition of $H$ we get
$$
\bar C_\ep (u,v) =  { H\lk(    U_1(\ep) u,U_2(\ep) v\rk)\over H\lk(  U_1(\ep),U_2(\ep )\rk)}
$$
In case $\nu_1=\nu_2$, we get by the scaling property of the Clayton copula (see Definition \ref{clayton})
\DEQS \bar C_\ep (u,v) =  {     U_1(\ep)
\over U_1 (\ep) \, H\lk(  1,1\rk)} H(u,v)= H(u,v).
\EEQS
This means that the survival copula $\bar C_\ep$ is given by $H$.

\begin{prop}
Let us assume the copula satisfies the following scaling property
\DEQSZ\label{scaling1}
 H(\alpha u_1,\alpha u_2)= \alpha H(u_1,u_2), \quad u_1,u_2\in\RR.
\EEQSZ
Let us define %
$$
h(u_1,u_2) := { {\partial ^ 2 \over \partial u_1\partial u_2} H(u_1,u_2)
 }\, . %\,{\nu_1(dz_1)}.
$$
Then the conditional probability of
$ Y_{1,\ep}$ given $Y_{2,\ep}$ is
 \DEQS
\PP\lk( Y_{1,\ep}=z_1\mid Y_{2,\ep}=z_2\rk)  & = &h(u_1,u_2)\Big|_{u_1=U_1 (z_1)\atop u_2=U_2 (z_2)} \,{ \nu(dz_1)},
 \EEQS
 for the case where $\nu_1=\nu_2$.
\end{prop}
\begin{proof}
The formula can be shown by direct calculations.
In particular, we can argue along the following lines
\DEQS
\PP\lk( Y_{1,\ep}=z_1\mid Y_{2,\ep}=z_2\rk)  = { \PP\lk( \{Y_{1,\ep} =z_1\}\cap \{ Y_{2,\ep} =z_2\}\rk)\over \PP(\{Y_{2,\ep}=z_2\})}
 = { {\partial ^ 2 \over \partial z_1\partial z_2} \bar F_\ep \lk(z_1,z_2\rk) \over {\partial  \over \partial z_2} \bar F _\ep(0,z_2) }.
 \EEQS
 Owing to the following equalities
 \DEQS
  {\partial  \over \partial z_2} \bar F _\ep(0,z_2)=  {\partial  \over \partial z_2} \bar F _{2,\ep}(z_2)={{\partial  \over \partial z_2} U_2(z_2)\over U_2(\ep)}={\nu_2(z_2)\over U_2(\ep)},
  \EEQS
and the scaling property \eqref{scaling1} we get
 \DEQS
\lqq{ \PP\lk( Y_{1,\ep}=z_1\mid Y_{\ep,2}=z_2\rk) }&&
\\& = &{ {\partial ^ 2 \over \partial u_1\partial u_2}  H \lk(u_1,u_2\rk)\Big|_{u_1=\bar F_{\ep,1} (z_1)\atop u_2=\bar F_{\ep,2} (z_2)}
%{\partial  \over \partial z_2} H (u_1,z_2) \Big|_{u_1=\bar F_{\ep,1}(0 )\atop u_2=\bar F_{\ep,2} (z_2)}
 }
{ {\partial \over \partial z_1} \bar F_{1,\ep}(z_1){\partial \over \partial z_2} \bar F_{2,\ep}(z_2)\over {\partial \over \partial z_2} \bar F_{2,\ep}(z_2)}
\\
& =& 2^ {1/\theta}{ {\partial ^ 2 \over \partial u_1\partial u_2}  H \lk(u_1,u_2\rk)\Big|_{u_1=U_1 (z_1)/U_1(\ep)\atop u_2=U_2 (z_2)/U_2(\ep)}
 } {\partial \over \partial z_1} \bar F_{1,\ep}(z_1)
\\
&= & 2^ {1/\theta}h(u_1,u_2)\Big|_{u_1=U_1 (z_1)\atop u_2=U_2 (z_2)} \,{ \nu_1(dz_1)}.
 \EEQS
\end{proof}

\begin{ex}
As mentioned in example \ref{clayton}, the Clayton copula is given by
$$
H(u_1,u_2) = \lk( \frac 12 u^ {-\theta}_1+\frac 12  u^ {-\theta}_2\rk) ^ {-\frac 1\theta} \beta 1_{u_1u_2>0}, \quad u_1,u_2\ge 0.
$$
A short calculation shows  that for $i=1,2$
$$
{\partial H(u_1,u_2)\over \partial u_i} =  \frac 12  \lk( \frac 12 u^ {-\theta}_1+\frac 12 u^ {-\theta}_2\rk) ^ {-\frac 1\theta-1} \, u_i^ {-\theta-1}
$$
and
$$
{\partial^2  H(u_1,u_2)\over \partial u_1 \partial u_2} = \frac 14  \lk(1+ \theta\rk)
\lk( \frac 12 u^ {-\theta}_1+\frac 12  u^ {-\theta}_2\rk) ^ {-\frac 1\theta-2} \, u_1^ {-\theta-1}u_2^ {-\theta-1}.
$$
Therefore
\DEQS
h(u_1,u_2) &=& \frac 14  {  \lk(1+\theta\rk)   \lk( \frac 12 u^ {-\theta}_1+\frac 12 u^ {-\theta}_2\rk) ^ {-\frac 1\theta-2} \, u_1^ {-\theta-1}u_2^ {-\theta-1}
  }\, ,
\EEQS
which implies that
 \DEQS \lqq{ \PP\lk( \Delta L_1(t)=z_1\mid \Delta L_2(t)
=z_2\rk)} &&
\\ & =&\frac 14  {  \lk(1+ \theta\rk)   \lk(\frac 12  u^ {-\theta}_1+\frac 12 u^ {-\theta}_2\rk) ^ {-\frac 1\theta-2} \, u_1^ {-\theta-1}
 u_2^ {-\theta-1}  }\,\Big|_{u_1=U_1(z_1)\atop u_2=U_2(z_2)} \nu_1(dz_1)
.
\EEQS

\end{ex}

\section{Application of \levy-Upward Theorem}\label{levy_upward}
\label{aC}

%In this section, we represent a theorem with is an application of  \levy-Upward Theorem used as a tool to prove the theorem \eqref{copula_inf}.
 Before we start our main theorem of this section, we will illustrate the following remark which is useful to complete the proof of the main theorem of this section.

\begin{rem}
 Let $(\Omega,\CA,\mu)$ be a measure space and $\CL\subset \CA$. We say that
$\CL$ is a lattice, if $\CL$ is closed under countable unions and intersections, and $\emptyset,\,\Omega\in \CL$.
 Let $\CL^ c:=\{ A\in \CA, \Omega\setminus A\in\CL\}$. By the definition of the $\sigma$--algebra, we know
that if $\CL$ is a $\sigma $--algebra, then   $\CL$ is also a lattice and $\CL^c=\CL$. Therefore, in case $\CL$ is a $\sigma$--algebra,
Theorem 3.1 in  \cite {Tim} reads:
$$
\EE_\gamma\lk[\frac 1X\mid \CL\rk]= \lk(\EE_\mu\lk[ X\mid \CL\rk]\rk)^ {-1},
$$
with $\gamma(A)=\int_{A}X(\omega)\mu(d\omega)$ and $X$ is a square integrable random variable.
\end{rem}
Using Theorem 3.1 of \cite{Tim} and the L\'evy's upward Theorem we can show the following Theorem. %By the remark we can show.
\begin{thm}\label{mainC}
Let $V=\{V(t):t\ge 0\}$ be a solution to equation  \eqref{eqn-vnoep}
and $V^\ep=\{V^\ep(t):t\ge 0\}$, $\ep\in(0,1]$, be  the family of a solutions to    \eqref{eqn-vep}. Let $\{\Gamma_\ep:\ep\in(0,1)\}$  be a family  of uniformly integrable %(U.I.) (see p.45, \cite{kallenberg})
 stochastic processes. Fix $p=1$ or $2$.
In particular, for any $t\ge 0$ the family $\{|\Gamma_\ep (t)|^ {4p}: \ep\in(0,1]\}$ is uniformly integrable and
$\lim_{\ep \to 0} \Gamma_\ep(t)=\Gamma(t)$, $\QQ$--a.s.
Then,  we have $\QQ$--a.s. and in $L^ 1(\Omega;\RR)$ %for any $p=1$ or $2$
\DEQSZ\label{levyconvergence}
\lim_{\ep\to 0}  \EE^ {\QQ}  \bigl| \EE^ {\QQ_\ep} \lk[\Gamma_  \ep(t) V^  \ep (t)\mid \CY_t^\ep\rk]\bigr|^ p = \EE^ {\QQ} \lk[ \Gamma(t) V (t)\mid \CY_t\rk] \bigr|^p,
\EEQSZ
where  \DEQS { d \QQ\over
d\QQ_\eps}\Big|_{\CF_t}&=&\frac{ V^\eps (t)}{ V(t)},\quad t\ge 0.\EEQS
\end{thm}

\begin{proof}
Apply the Kallianpur-Striebel formula to get
\DEQSZ\label{specialupward}
&&  \EE^ {\QQ}  \bigl| \EE^ {\QQ_\ep} \lk[\Gamma_  \ep(t) V^  \ep (t)\mid \CY_t^\ep\rk] - \EE^ {\QQ} \lk[ \Gamma(t) V (t)\mid \CY_t\rk] \bigr|^p
\\
&& {} = \EE^ {\QQ}  \Bigl|\, \frac{\EE^ {\QQ} \lk[\Gamma_  \ep(t) V (t)\mid \CY_t^\ep\rk]}{\EE^ {\QQ} \lk[\frac{ V (t)}{ V^\eps(t)}\mid \CY_t^\ep\rk]} - \EE^ {\QQ} \lk[ \Gamma(t) V (t)\mid \CY_t\rk] \,\Bigr|^p
\nonumber
\\
&\le&  2^{p-1}\EE^ {\QQ}  \Bigl|\,  \frac{\EE^ {\QQ} \lk[\Gamma_  \ep(t) V (t)\mid \CY_t^\ep\rk]}{\EE^ {\QQ} \lk[\frac{ V (t)}{ V^\eps(t)}\mid \CY_t^\ep\rk]} -  \frac{\EE^ {\QQ} \lk[\Gamma(t) V (t)\mid \CY_t\rk]}{\EE^ {\QQ} \lk[\frac{ V (t)}{ V^\eps(t)}\mid \CY_t^\ep\rk]}\,\Bigr|^p
\nonumber
\\
&+&  2^{p-1}\EE^ {\QQ}  \Bigl|\frac{\EE^ {\QQ} \lk[\Gamma(t) V (t)\mid \CY_t\rk]}{\EE^ {\QQ} \lk[\frac{ V (t)}{ V^\eps(t)}\mid \CY_t^\ep\rk]} - \EE^ {\QQ} \lk[ \Gamma(t) V (t)\mid \CY_t\rk] \Bigr|^p.
\nonumber
\EEQSZ
The H\"older inequality gives
\DEQS
\ldots  && {} \le  2^{p-1}\Bigl( \EE^ {\QQ}  \Bigl| \frac{1}{\EE^ {\QQ} \lk[\frac{ V (t)}{ V^\eps(t)}\mid \CY_t^\ep\rk]} \, \Bigr|^{2p}\Bigr)^ \frac 12 \lk(  \EE^ {\QQ}  \Bigl| \EE^ {\QQ} \lk[\Gamma_  \ep(t) V (t)\mid \CY_t^\ep\rk] -  \EE^ {\QQ} \lk[\Gamma(t) V (t)\mid \CY_t\rk]\, \Bigr|^{2p}\rk) ^ \frac 12
\\
&&{}+2^{p-1}\Bigl( \EE^ {\QQ}  \Bigl| \frac{  \EE^ {\QQ} \lk[\Gamma(t) V (t)\mid \CY_t\rk]}{\EE^ {\QQ} \lk[\frac{ V (t)}{ V^\eps(t)}\mid \CY_t^\ep\rk]} \,\Bigr|^{2p}\, \Bigr)^ \frac 12 \Bigl(\, \EE^ {\QQ}  \Bigl| \EE^ {\QQ} \lk[\frac{ V (t)}{ V^\eps(t)}\mid \CY_t^\ep\rk] -  1\, \Bigr|^{2p}\, \Bigr)^\frac 12 .
\EEQS
Now we will show that for $\ep\to 0$, the first term in last inequality converges  to zero. First, we will show that there exists a constant $C>0$ such that
%In order to show that first  we need to establish the term
$$
\EE^ {\QQ}  \lk| \frac{1}{\EE^ {\QQ} \lk[\frac{ V (t)}{ V^\eps(t)}\mid \CY_t^\ep\rk]} \rk|^{2p}<C,\quad \ep\in(0,1].
$$
 By Theorem 3.1 in  \cite {Tim}, Jensen's inequality and H\"older inequality we get

\DEQSZ\label{epbound}
\EE^ {\QQ}  \lk| \frac{1}{\EE^ {\QQ} \lk[\frac{ V (t)}{ V^\eps(t)}\mid \CY_t^\ep\rk]} \rk|^{2p}& =&  \EE^ {\QQ}  \lk| \EE^ {\QQ_\ep} \lk[\frac{ V^\eps (t)}{ V(t)}\mid \CY_t^\ep\rk] \rk|^{2p}
\nonumber\\&\le&   \EE^ {\QQ} \lk( \EE^ {\QQ_\ep} \lk[\lk| \frac{ V^\eps (t)}{ V(t)} \rk|^{2p}\mid \CY_t^\ep\rk]\rk)
\le   \EE^ {\QQ_\ep}\lk( \frac{ V^\eps (t)}{ V(t)} \EE^ {\QQ_\ep} \lk[\lk| \frac{ V^\eps (t)}{ V(t)} \rk|^{2p}\mid \CY_t^\ep\rk]\rk)
\nonumber\\&\le&   \lk(\EE^ {\QQ_\ep}\lk| \frac{ V^\eps (t)}{ V(t)} \rk|^2\rk)^\frac 12 \lk( \EE^ {\QQ_\ep}\lk| \frac{ V^\eps (t)}{ V(t)} \rk|^{4p}\rk)^\frac 12
 \nonumber\\&=&  \lk(\EE^ {\QQ}\lk| \frac{ V^\eps (t)}{ V(t)} \rk|\rk)^\frac 12  \lk(\EE^ {\QQ}\lk| \frac{ V^\eps (t)}{ V(t)} \rk|^{4p-1}\rk)^\frac 12. %<\infty.
\EEQSZ
To see that the last terms are bounded, first, note that $V^ {-1}=Z$ where $Z$ solves \eqref{zsolves}. Due to the fact that $g$ is bounded,  $Z$ has bounded moments of order $8p-2$.
In addition, for any $t\ge 0$, $V(t)$ and $V^ \eps(t)$ have  also uniform bounds  of order $8p-2$.
Hence, we conclude the RHS above is uniformly for all $\ep>0$ bounded.
%Applying Theorem 3.1 in  \cite {Tim} we  infer that  $ \frac{ V^\eps (t)}{ V(t)} $ is a positive random variable for each fixed $t>0$.

\medskip

Next,  we would like to show that
\DEQSZ\label{epz}
\lim_{\ep\to 0} \EE^ {\QQ}  \lk| \EE^ {\QQ} \lk[\Gamma_  \ep(t) V (t)\mid \CY_t^\ep\rk] -  \EE^ {\QQ} \lk[\Gamma(t) V (t)\mid \CY_t\rk]\rk|^{2p}=0.
\EEQSZ

  For the notational  convenient, take $\tilde \Gamma_t^\ep= \Gamma_  \ep(t) V (t)$ and  $\tilde \Gamma_t^0= \Gamma(t) V (t)$. For fixed positive $R>0$ (the exact value of $R$ we will fix later) we get
\DEQSZ\label{epzero}
\EE^ {\QQ}  \lk| \EE^ {\QQ} \lk[\Gamma_t^\ep\mid \CY_t^\ep\rk] -  \EE^ {\QQ} \lk[ \Gamma_t^0\mid \CY_t\rk]\rk|^{2p}& \le&2^{p-1}\EE^ {\QQ}  \lk| \EE^ {\QQ} \lk[\Gamma_t^\ep\textbf{1}_{|\Gamma_t^\ep|\le R}\mid \CY_t^\ep\rk] -  \EE^ {\QQ} \lk[ \Gamma_t^0\textbf{1}_{|\Gamma_t^0|\le R}\mid \CY_t\rk]\rk|^{2p}
\nonumber\\&+&   2^{p-1}\EE^ {\QQ}  \lk| \EE^ {\QQ} \lk[\Gamma_t^\ep\textbf{1}_{|\Gamma_t^\ep|> R}\mid \CY_t^\ep\rk] -  \EE^ {\QQ} \lk[ \Gamma_t^0\textbf{1}_{|\Gamma_t^0|> R}\mid \CY_t\rk]\rk|^{2p}
\nonumber\\& \le&2^{p-1}\EE^ {\QQ}  \lk| \EE^ {\QQ} \lk[\Gamma_t^\ep\textbf{1}_{|\Gamma_t^\ep|\le R}\mid \CY_t^\ep\rk] -  \EE^ {\QQ} \lk[ \Gamma_t^0\textbf{1}_{|\Gamma_t^0|\le R}\mid \CY_t\rk]\rk|^{2p}
\nonumber\\&+&  2^{2p-2} \EE^ {\QQ}  \lk|\Gamma_t^\ep\textbf{1}_{|\Gamma_t^\ep|> R}\rk|^{2p} +  2^{2p-2}\EE^ {\QQ} \lk| \Gamma_t^0\textbf{1}_{|\Gamma_t^0|> R}\rk|^{2p}.
\EEQSZ
The last inequality holds due to the Jensen's inequality.
Since for any $t\ge 0$, the family $\{|\Gamma_t^\ep|^ {2p}:\ep\in(0,1]\}$ is uniformly integrable,  for any $\kappa>0$ there exist a
number $R>0$ such that for all $\ep\in(0,1]$,
$$ \EE^ {\QQ}  \lk|\Gamma_t^\ep\textbf{1}_{|\Gamma_t^\ep|> R}\rk|^{2p}<\frac{\kappa}{4}
$$
 and
$$  \EE^ {\QQ} \lk| \Gamma_t^0\textbf{1}_{|\Gamma_t^0|> R}\rk|^{2p}<\frac{\kappa}{4}.
$$
Let $R>0$ be fixed.
First,
\DEQS
\lqq{ \EE^ {\QQ}  \lk| \EE^ {\QQ} \lk[\Gamma_t^\ep\textbf{1}_{|\Gamma_t^\ep|\le R}\mid \CY_t^\ep\rk] -  \EE^ {\QQ} \lk[ \Gamma_t^0\textbf{1}_{|\Gamma_t^0|\le R}\mid \CY_t\rk]\rk|^{2p}} &&
\\
&&\le
R^ {2p-1} \EE^ {\QQ}  \lk| \EE^ {\QQ} \lk[\Gamma_t^\ep\textbf{1}_{|\Gamma_t^\ep|\le R}\mid \CY_t^\ep\rk] -  \EE^ {\QQ} \lk[ \Gamma_t^0\textbf{1}_{|\Gamma_t^0|\le R}\mid \CY_t\rk]\rk|.
\EEQS
By the \levy--Upward Theorem (see p.\ 196 in \cite{Dembo}), there exist a number $\ep_1>0$,
such that  for all $\ep\in(0,\ep_1]$,
$$
R^ {2p-1}\EE^ {\QQ}  \lk| \EE^ {\QQ} \lk[\Gamma_t^\ep\textbf{1}_{|\Gamma_t^\ep|\le R}\mid \CY_t^\ep\rk] -  \EE^ {\QQ} \lk[ \Gamma_t^0\textbf{1}_{|\Gamma_t^0|\le R}\mid \CY_t\rk]\rk|<\frac{\kappa}{2}.
$$
This implies that  for all $\ep\in (0,\ep_1]$,
$$
\EE^ {\QQ}  \lk| \EE^ {\QQ} \lk[\Gamma_t^\ep\mid \CY_t^\ep\rk] -  \EE^ {\QQ} \lk[ \Gamma_t^0\mid \CY_t\rk]\rk|^{2p}<\kappa.
$$
This gives Claim \eqref{epz}. Combining results \eqref{epbound} and \eqref{epz}, implies that  the first term in last inequality of \eqref{specialupward} goes to zero as $\ep\to 0$. It remains to show
$$\Bigl|\EE^ {\QQ} \lk[\frac{ V (t)}{ V^\eps(t)}\mid \CY_t^\ep\rk] -  1\, \Bigr|^{2p} \rightarrow 0\,\mbox{ as } \ep\to 0.$$
By similar arguments we can  prove that
the term above also converges
to zero as $\ep\to 0$, which gives the assertion.
\end {proof}

\section*{Acknowledgements}
{The authors are very thankful to the anonymous referee for his/her insightful comments and remarks, which improve the manuscript. We are also very grateful to Professor Dan Crisan for his valuable suggestions and comments for improving the manuscript. This work was supported by the Austrian Science foundation (FWF), Projectnumber
P17273-N12.}

%

%\end{document}

 \end{document}